\documentclass[reqno,a4paper,11pt]{amsart}

\usepackage[all,poly]{xy}
\usepackage{amsfonts,stmaryrd}
\usepackage[mathcal]{eucal}
\usepackage{amssymb}
\usepackage{amsmath}
\usepackage{mathrsfs}
\usepackage{stmaryrd}
\usepackage{color}
\usepackage[pagebackref,colorlinks]{hyperref}
\usepackage{enumerate}

\theoremstyle{plain}
\newtheorem {Lem}{Lemma}
\newtheorem {The}{Theorem}

\newtheorem {Prob}{Problem}

\theoremstyle{remark}

\newtheorem {Exp}[Lem]{Example}

\theoremstyle{definition}

\newcommand\Label[1]{\label{#1}}

\newcommand{\ep}{\varepsilon}

\newcommand{\gm}{\mathfrak m}

\newcommand{\GL}{\operatorname{GL}}
\newcommand{\GU}{\operatorname{GU}}
\newcommand{\EU}{\operatorname{EU}}
\newcommand{\FU}{\operatorname{FU}}
\newif\ifcomm
\let\ifcomm\iffalse

\newcommand{\FormR}{A,\Lambda}

\newcommand{\FUnT}[2]{\FU(2n,t^{#1}A,t^{#1}#2,t^{#1}\Gamma)} 
\newcommand{\FUnTJ}[2]{\FU(2n,t^{#1}A,t^{#1}#2,t^{#1}\Delta)}
\newcommand{\FUnt}[2]{\FU(2n,t^{#1}#2,t^{#1}\Gamma)} 

\newcommand{\FFUnt}[1]{\FU(2n,t^{#1}A,t^{#1}\Lambda)}

\newcommand{\EUnT}[2]{\FU(2n,t^{#1}#2,t^{#1}#2,t^{#1}\Lambda)}
\newcommand{\EUnt}[2]{\FU(2n,t^{#1}#2,t^{#1}\Lambda)}

\newcommand{\FUnTk}[3]{\FU^{#3}(2n,t^{#1}A,t^{#1}#2,t^{#1}\Gamma)} 

\def\Sp{\operatorname{Sp}}

\def\SO{\operatorname{SO}}

\def\GL{\operatorname{GL}}

\def\SU{\operatorname{SU}}
\def\GU{\operatorname{GU}}
\def\EU{\operatorname{EU}}
\def\FU{\operatorname{FU}}
\def\CU{\operatorname{CU}}
\def\K{\operatorname{K}}

\def\Cent{\operatorname{Cent}}
\def\Max{\operatorname{Max}}

\def\sr{\operatorname{sr}}

\def\map{\longrightarrow}
\def\bar{\overline}
\def\epsilon{\varepsilon}
\def\e{\varepsilon}
\def\a{\alpha}

\def\Ga{\Gamma}

\title{Multiple Commutator Formulas for Unitary Groups}

\author{R.~Hazrat}
\address{
School of Computing and Mathematics\\
University of Western Sydney\\
Australia}
\email{r.hazrat@uws.edu.au}

\author{N.~Vavilov}
\address{Department of Mathematics and Mechanics,
St.~Petersburg State University, St.~Petersburg, Russia}
\email{nikolai-vavilov@yandex.ru}

\author{Z.~Zhang}
\address{Department of  Mathematics, Beijing Institute
of Technology, Beijing, China}
\email{zuhong@gmail.com}

\begin{thanks}
{The second author started this research within the framework
of the RFFI/Indian Academy cooperation project 10-01-92651
``Higher composition laws, algebraic $K$-theory and algebraic
groups'' (SPbGU--Tata Institute) and the RFFI/BRFFI cooperation
project 10-01-90016 ``The structure of forms of reductive
groups, and behaviour of small unipotent elements in
representations of algebraic groups'' (SPbGU--Mathematics
Institute of the Belorussian Academy). Currently his work is
supported by the RFFI research projects 11-01-00756 (RGPU) and
12-01-00947 (POMI), by the State Financed research task
6.38.74.2011 at the Saint Petersburg State University
``Structure theory and geometry of algebraic groups and their
applications in representation theory and algebraic $K$-theory''
and by the Presidential Grant 6.10.61.2012 for the leading
scientific schools. The third author acknowledges the support
of NSFC (Grant 10971011).}
\end{thanks}

\begin{document}

\begin{abstract}
Let $(\FormR)$ be a form ring such that $A$ is quasi-finite
$R$-algebra (i.e., a direct limit of module finite algebras) with
identity. We consider the hyperbolic Bak's unitary groups
$\GU(2n,\FormR)$, $n\ge 3$. For a form ideal $(I,\Gamma)$
of the form ring $(\FormR)$ we denote by $\EU(2n,I,\Gamma)$
and $\GU(2n,I,\Gamma)$ the relative elementary group and the
principal congruence subgroup of level $(I,\Gamma)$, respectively.
Now, let $(I_i,\Gamma_i) $, $i=0,...,m$, be form ideals of the
form ring $(A,\Lambda)$. The main result of the present paper is
the following multiple commutator formula
\[
\begin{split}
\big[\EU(2n,I_0,\Gamma_0),&\GU(2n,I_1,\Gamma_1),\GU(2n, I_2,\Gamma_2),
\ldots, \GU(2n,I_m,\Gamma_m)\big]=\\
&\big[\EU(2n,I_0,\Gamma_0),\EU(2n,I_1,\Gamma_1),\EU(2n,I_2,\Gamma_2),
\ldots, \EU(2n, I_m, \Gamma_m)\big],
\end{split}
\]
which is a broad generalization of the standard commutator formulas.
This result contains all previous results on commutator formulas
for classical like-groups over commutative and finite-dimensional
rings.
\end{abstract}

\maketitle

\section{Introduction}\label{sec1}

Let $(\FormR)$ be a form ring, $n\ge 3$ and let $\GU(2n,\FormR)$
be the hyperbolic Bak's unitary group \cite{HO,knus,BV3,RN}.
In the paper \cite{RNZ1} we obtained relative commutator formulas
for the unitary groups $\GU(2n,\FormR)$, under some natural
commutativity/finiteness assumptions on $(\FormR)$. The goal of
the present paper is to enhance the relative localisation method
developed in \cite{RNZ1} and to prove {\it multiple\/} relative
commutator formulas, which serve as a simultaneous generalisation
of all previously known such results. For the general linear group
$\GL(n,A)$ similar results were recently established by the
first and the third author in \cite{RHZZ2}.
\par
Actually, since the general linear group is a special case of
Bak's unitary group, the results of the present paper are not
only modeled on \cite{RHZZ2} but also generalise the results of
\cite{RHZZ2}. Our results are new already in the following
classical situations.
\par\smallskip
$\bullet$ The case of symplectic groups $\Sp(2l,R)$, where the
involution is trivial, and $\Lambda=R$.
\par\smallskip
$\bullet$ The case of even split orthogonal groups $\SO(2l,R)$,
where the involution is trivial and $\Lambda=0$.
\par\smallskip
$\bullet$ The case of classical unitary groups $\SU(2l,R)$,
where $\Lambda=\Lambda_{\max}$.
\par\smallskip
As the proofs in \cite{RHZZ2}, the proofs in the present paper
are based on a version of localisation. The two most familiar
versions of localisation are Quillen--Suslin's localisation and
patching, and Bak's localisation-completion. Actually, in this
paper we use a version of Bak's method \cite{B4}, which was first
applied to unitary groups in the Bielefeld Thesis of the first author
\cite{RH,RH2}. It is interesting to note that in this generality
the first convincing treatment of Quillen--Suslin's localisation and
patching method appeared only afterwards, in Petrov's Saint Petersburg
Thesis \cite{petrov1,petrov2,petrov3}. We do not attempt to give
an account of the historical development of localisation methods,
see our surveys \cite{RN,yoga} for more details and many related
references.
\par
More precisely, our proofs rely on a further enhancement of the
{\it relative\/} localisation method introduced by the first and
the third author \cite{RHZZ1} in the context of the general
linear group, and applied to Bak's unitary groups in \cite{RNZ1}
and to Chevalley groups in \cite{RNZ2}. Initially, this method
was proposed to address problems raised by Alexei Stepanov and
the second author \cite{NVAS}. See our published papers
\cite{RN1,BRN,RHZZ1,RNZ1,SV10,RHZZ2} and our forthcoming papers
\cite{RNZ2,yoga2,portocesareo,HSVZmult,HSVZ} for many
further recent applications of this method and other offsprings
of Bak's method, including the remarkable universal localisation
by Alexei Stepanov \cite{stepanov10}. Compare also the recent
papers by Anthony Bak, Rabeya Basu, Khanna, Alexander Luzgarev,
Victor Petrov, Ravi Rao, Anastasia Stavrova and Matthias Wendt
\cite{petrov1,petrov2,BRK,Basu06,BBR,PS08, Basu11,wendt,
luzgarevstavrova,stavrova}, for latest versions and fresh
applications of Quillen--Suslin's method.
\par
Since the present paper is an immediate sequel of \cite{RHZZ2,RNZ1},
we do not reproduce a detailed historical survey of the commutator
formulas, and do not discuss crucial early contributions by Hyman
Bass \cite{Bass2,Bass1}, Anthony Bak \cite{B1,bass73},
Andrei Suslin and Vyacheslav
Kopeiko \cite{SUS,suslinkopeiko,kopeiko,TUL}, Alec Mason and Wilson
Stothers
\cite{MAS3,Mason74,MAS1,MAS2}, Leonid Vaserstein \cite{V1,vaser86},
Zenon Borewicz and the second author \cite{borvav}, Giovanni Taddei
\cite{taddei}, and others.
Instead, we refer to our surveys \cite{NV91,RN} and to the
papers \cite{BV3,ASNV} for an accurate historic description and
many further references.
\par
However, to put our results in context, let us briefly review the
standard commutator formulas for unitary groups. Below,
$\EU(2n,\FormR)$ denotes the [absolute] elementary unitary group,
generated by the elementary root unipotents. As usual,
for a form ideal $(I,\Gamma)$ of the form ring $(\FormR)$ we
denote by $\EU(2n,I,\Gamma)$ the corresponding relative elementary
subgroup, by $\GU(2n,I,\Gamma)$ the principal congruence
subgroup of level $(I,\Gamma)$ and by $\CU(2n,I,\Gamma)$ the
full congruence subgroup of level $(I,\Gamma)$.
\par
One of the main results of Bak's Thesis \cite{B1}, Theorem 1.2
can be summarised as follows. Actually, the second formula is
not part of the statement of that theorem, but it appears in its
proof, at the bottom of page 4.2, see corollary 3.4 on page 3.22
of \cite{B1}, or \cite{bass73}. The group $\CU(2n,I,\Gamma)$
is defined differently, but from \cite{BV3} we know that in all
interesting situations, including the ones covered by Theorems~\ref{main_Bak}
 and \ref{main_BV} below, all definitions of these groups coincide.

\begin{The}[{Bak}]\label{main_Bak}
Let $R$ be a Noetherian commutative ring of Bass--Serre
dimension $d$ and let $(\FormR)$ be a form ring module finite
over $R$-algebra. Assume that $n\ge d+1,3$. Further, let
$(I,\Gamma)$ be a form ideal of the form ring $(\FormR)$. Then
\begin{equation}
\big[\GU(2n,A,\Lambda),\EU(2n,I,\Gamma)\big]=
\big[\EU(2n,A,\Lambda),\CU(2n,I,\Gamma)\big]=
\EU(2n,I,\Gamma).
\end{equation}
\end{The}

The following result is referred to as the {\it absolute standard
commutator formula\/}. In this generality it is established by
Anthony Bak and the second author \cite{BV2,BV3} and by Leonid
Vaserstein and You Hong \cite{VY}.\footnote{The paper \cite{VY}
uses a naive form of reduction modulo a form ideal, instead of
the correct form proposed in \cite{B1}, see also \cite{BV3}.
Thus, strictly speaking, the proofs in \cite{VY} are only valid
when $\Lambda=\Lambda_{\max}$. Still, it is independent of
\cite{BV2} and at the time was a non-trivial contribution to
our understanding of the structure of unitary groups.}
This result, and the more general Theorem~\ref{main_HVZ} will be used
throughout the present paper.

\begin{The}[{Bak--Vavilov, Vaserstein--You Hong}]\label{main_BV}
Let $n\ge 3$, $R$ be a commutative ring, $(\FormR)$ be a form ring such
that $A$ is a quasi-finite $R$-algebra. Further, let $(I,\Gamma)$
be a form ideal of the form ring $(\FormR)$. Then
\begin{equation}
\big[\GU(2n,A,\Lambda),\EU(2n,I,\Gamma)\big]=
\big[\EU(2n,A,\Lambda),\CU(2n,I,\Gamma)\big]=
\EU(2n,I,\Gamma).
\end{equation}
\end{The}

In the context of Bak's unitary groups, the history of relative
standard commutator formula starts with the Thesis of G\"unter
Habdank \cite{Ha1}, see also \cite{Ha2}. The following result
is essentially \cite{Ha2}, Proposition 3.5. Actually, there it
is stated in terms of a certain rather technical quadratic stable
rank condition. Not to recall its definition here, we limit
ourselves with a special case of this result, under the same
assumption as Bak's theorem.

\begin{The}[{Habdank}]\label{main_Hab}
Let $R$ be a Noetherian commutative ring of Bass--Serre
dimension $d$ and let $(\FormR)$ be a form ring module finite
over $R$-algebra. Assume that $n\ge d+1,3$. Further, let $(I,\Gamma)$ and
$(J,\Delta)$ be two form ideals of the form ring $(\FormR)$. Then
\begin{equation}
\big[\EU(2n,I,\Gamma),\GU(2n,J,\Delta)\big]=
\big[\EU(2n,I,\Gamma),\EU(2n,J,\Delta)\big].
\end{equation}
\end{The}

For unitary groups, the following result was proven in our
previous paper \cite{RNZ1}. Both this result itself, and the
methods used in its proof are instrumental throughout the
present paper. Actually, it is the induction base of
our main theorem, and will be repeatedly invoked in its
proof.
\par
It is modeled on early contributions by Alec Mason and Wilson Stothers
\cite{MAS1,Mason74,MAS2,MAS3}. For the general linear group we gave
{\it three\/} independent proofs of a similar result:
Stepanov--Vavilov \cite{NVAS}, based on decomposition of unipotents,
Hazrat--Zhang \cite{RHZZ1}, based on localisation, and
Stepanov--Vavilov \cite{NVAS2}, based on the absolute commutator
formula and level calculations.
Unfortunately -- or, maybe, fortunately! -- at that time we
were not aware of the {\it extremely\/} important paper by
You Hong \cite{you92}, where a similar result was obtained for
Chevalley groups, with a proof very close to the {\it second\/}
proof by Stepanov and the second author, \cite{NVAS2}. For
otherwise we would not be as eager to develop a localisation
proof.

\begin{The}[{Hazrat--Vavilov--Zhang}]\label{main_HVZ}
Let $n\ge 3$, $R$ be a commutative ring, $(\FormR)$ be a form ring such
that $A$ is a quasi-finite $R$-algebra. Further, let $(I,\Gamma)$ and
$(J,\Delta)$ be two form ideals of the form ring $(\FormR)$. Then
\begin{equation}\Label{eq:final}
\big[\EU(2n,I,\Gamma),\GU(2n,J,\Delta)\big]=
\big[\EU(2n,I,\Gamma),\EU(2n,J,\Delta)\big].
\end{equation}
\end{The}

Such was the state of art before the present paper. Here, we
generalise all these formulas to an arbitrary number of form ideals.
The main result of the present paper may be stated as follows.
For a start, multiple commutators can be interpreted as left
normed commutators.

\begin{The}\Label{main_multi}
Let $n\ge 3$, $R$ be a commutative ring, $(\FormR)$ be a form ring such
that $A$ is a quasi-finite $R$-algebra. Furthermore, let $(I_i,\Gamma_i)$,
$i=0,...,m$, be form ideals of the form ring $(\FormR)$. Then
\begin{multline*}
\big[\EU(2n,I_0,\Gamma_0),\GU(2n,I_1,\Gamma_1),\GU(2n,I_2,\Gamma_2),
\ldots,\GU(2n,I_m,\Gamma_m)\big]=\\
=\big[\EU(2n,I_0,\Gamma_0),\EU(2n,I_1,\Gamma_1),\EU(2n,I_2,\Gamma_2),
\ldots,\EU(2n,I_m,\Gamma_m)\big].
\end{multline*}
\end{The}

This result is interesting in itself, but
its true significance is that it is absolutely indispensable
to proceed to the proof of the {\it general\/} multiple
commutator formula, which simultaneously generalises both
the standard commutator formulas and the nilpotent
structure of relative $\K_1$ established in \cite{RH,RH2,BRN}.
Using a whole bunch of difficult external results the authors
and Alexei Stepanov \cite{HSVZmult} have been able to establish such a
general multiple commutator formula for the case of $\GL(n,R)$,
but for other groups many tools are still missing, and Theorem~\ref{main_multi}
bridges one of these gaps.
\par
Multiple formula is also relevant as a prerequisite for the
description of subnormal subgroups of unitary groups. See \S\ref{sec11}
for further comments on these and other possible applications.
\par
Actually, in \S\ref{sec10} we prove a still more general result, where both
the position of the elementary factor in the left hand side, and the
arrangement of brackets may be arbitrary. In fact, Theorem~\ref{main_multi}
almost immediately implies the following result.

\begin{The}\Label{main_multi2}
Let $n\ge 3$, $R$ be a commutative ring, $(\FormR)$ be a form ring such
that $A$ is a quasi-finite $R$-algebra. Furthermore, let $(I_i,\Gamma_i)$,
$i=0,...,m$, be form ideals of the form ring $(\FormR)$ and $G_i$ be
subgroups of $\GU(2n,\FormR)$ such that
$$ \EU(2n,I_i,\Gamma_i)\subseteq G_i\subseteq \GU(2n,I_i,\Gamma_i),
\quad\text{ for } i=0,\ldots, m. $$
\noindent
If there is an index $j$ such that $G_j=\EU(2n,I_j,\Gamma_j)$, then
\begin{equation}\label{hhalkkk}
\big \llbracket G_0,G_1,\ldots, G_m\big \rrbracket=\big \llbracket
E(I_0),E(I_1),\ldots,E(I_m) \big \rrbracket.
\end{equation}
\end{The}

Observe though, that
the arrangement of brackets in~(\ref{hhalkkk}) should be the same on both sides,
the mixed commutators are {\it not\/} associative! In particular,
there are easy counter-examples which show that in general
\begin{multline*}
\Big[\big[\EU(2n,I,\Gamma),\EU(2n,J,\Delta)\big],\EU(2n,K,\Omega)\Big]\neq\\
\Big[\EU(2n,I,\Gamma),\big[\EU(2n,J,\Delta),\EU(2n,K,\Omega)\big]\Big].
\end{multline*}
\par
Actually, the difficult part of the proof of Theorem~\ref{main_multi}, which
allows to carry through inductive step is the following
{\it triple\/} commutator formula. It is precisely the proof
of that special case that requires new ideas, as compared with
\cite{RNZ1}, and entails most of the technical strain.

\begin{The}\label{main_triple}
Let $n\ge 3$, $R$ be a commutative ring, $(\FormR)$ be a form ring such
that $A$ is a quasi-finite $R$-algebra. Further, let $(I,\Gamma)$,
$(J,\Delta)$ and $(K,\Omega)$ be three form ideals of a form ring
$(\FormR)$. Then
\begin{multline*}
\Big[\big[\EU(2n,I,\Gamma),\GU(2n,J,\Delta)\big],\GU(2n,K,\Omega)\Big]=\\
\Big[\big[\EU(2n,I,\Gamma),\EU(2n,J,\Delta)\big],\EU(2n,K,\Omega)\Big].
\end{multline*}
\end{The}

In turn, modulo the standard commutator formula the proof of Theorem~\ref{main_triple}
amounts to the proof of the following equality:
\begin{multline*}
\Big[\big[\EU(2n,I,\Gamma),\EU(2n,J,\Delta)\big],\GU(2n,K,\Omega)\Big]=\\
\Big[\big[\EU(2n,I,\Gamma),\EU(2n,J,\Delta)\big],\EU(2n,K,\Omega)\Big].
\end{multline*}
\noindent
Essentially, the proof of this last equality is the technical core
of the present paper. It cannot be established with the use of the
relative commutator calculus developed in our paper \cite{RNZ1}.
In fact, to prove it we have to develop another layer of the commutator
calculus, which works with the generators of
$\big[\EU(2n,I,\Gamma),\EU(2n,J,\Delta)\big]$ rather than with
the usual elementary generators.
\par
Let us mention an {\it amazing\/} corollary of our results. It shows
that {\it any\/} multiple commutator of relative elementary groups
always equals a {\it double\/} mixed commutator, for some other
form ideals. In the following lemma $(I,\Gamma)\circ(J,\Delta)$
denotes the symmetrised product of form ideals, whose definition is
recalled in \S\ref{sec2}. This product is not associative, the bracketing
of the form ideals on the right hand side should correspond to the
bracketing of commutators on the left-hand side.

\begin{The}\Label{main_multi3}
Let $n\ge 3$, $R$ be a commutative ring, $(\FormR)$ be a form ring such
that $A$ is a quasi-finite $R$-algebra. Furthermore, let $(I_i,\Gamma_i)$,
$i=0,...,m$, be form ideals of the form ring $(\FormR)$. Consider an
arbitrary configuration of brackets $[\![\ldots]\!]$ and assume that
the outermost pair of brackets between positions $k$ and $k+1$. Then
\begin{multline*}
\big\llbracket\EU(2n,I_0,\Gamma_0),\EU(2n,I_1,\Gamma_1),\EU(2n,I_2,\Gamma_2),
\ldots,\EU(2n,I_m,\Gamma_m)\big\rrbracket=\\
=\Big[\EU\big(2n,(I_0,\Gamma_0)\circ\ldots\circ(I_k,\Gamma_k)\big),
\EU\big(2n,(I_{k+1},\Gamma_{k+1})\circ\ldots\circ(I_m,\Gamma_m)\big)\Big].
\end{multline*}
\end{The}

As opposed to that, in general the double commutators
$\big[\EU(2n,I,\Gamma),\EU(2n,J,\Delta)\big]$ do not coincide with
the elementary subgroups $\EU(2n,(I,\Gamma)\circ(J,\Delta))$. This is
indeed the case, when $I$ and $J$ are comaximal, $I+J=A$, but without
this additional assumption there are counter-examples even for such
nice rings as Dedekind rings of arithmetic type.

\par
The paper is organised as follows. In \S\S\ref{sec2}--\ref{sec4} we recall basic
notation, and some background facts concerning form rings and form
ideals, Bak's unitary groups and their relative subgroups, on which
the rest of the paper relies. The rest of the paper is devoted to
detailed proofs of the above Theorems~\ref{main_multi} and~\ref{main_multi2}. In \S\ref{sec5} and \S\ref{sec6}
we prove two important general results of technical nature, which
improve and elaborate the results of \cite{BV3} and \cite{RNZ1}.
Namely, in \S\ref{sec5} we finalise the calculation of levels for the
mixed commutator subgroups $\big[\GU(2n,I,\Gamma),\GU(2n,J,\Delta)\big]$,
whereas in \S\ref{sec6} we construct a generating system of the mixed
elementary commutator subgroups
$\big[\EU(2n,I,\Gamma),\EU(2n,J,\Delta)\big]$. The next
three sections constitute the technical core of the paper. Namely, in
\S\ref{sec7} and \S\ref{sec8} we unfold another layer of the relative commutator
calculus, with the elementary generators being replaced by our new
generators of the mixed commutator subgroups. This brings us to the
stage, where we can carry through the usual patching procedure to prove
the {\it triple\/} relative commutator formula. This is accomplished
in \S\ref{sec9}. At this point, we are almost there: the rest follows from
the double and the triple formulas and level calculations by the
standard group theoretic arguments. Not to repeat these routine
arguments in future, in \S\ref{sec10} we do this part of the proof axiomatically,
for all group functors enjoying some formal properties. Finally, in
\S\ref{sec11} we indicate some further possible applications of our results
and state some unsolved problems.


\section{Form rings and form ideal}\label{sec2}

The notion of $\Lambda$-quadratic forms, quadratic modules and generalised
unitary groups over a form ring $(A,\Lambda)$ were introduced by Anthony
Bak in his Thesis, see \cite{B1,B2}. In this section, and the next one, we
{\it very briefly\/} review the most fundamental notation and results
that will be constantly used in the present paper. We refer to
\cite{B2,HO,knus,BV3,RH,RH2,RN,tang,lavrenov} for details, proofs, and
further references.


\subsection{}\label{form algebra}
Let $R$ be a commutative ring with $1$, and $A$ be an (not necessarily
commutative) $R$-algebra. An involution, denoted by $\bar{\phantom{\a}}$, 
is an
anti-homomorphism of $A$ of order $2$. Namely, for $\alpha,\beta\in A$,
one has $\overline{\alpha+\beta}=\bar\alpha+\bar\beta$, \
$\overline{\alpha\beta}=\bar\beta\bar\alpha$ and $\bar{\bar\alpha}=\alpha$.
Fix an element $\lambda\in\Cent(A)$ such that $\lambda\bar\lambda=1$. One may
define two additive subgroups of $A$ as follows:
$$ \Lambda_{\min}=\{\alpha-\lambda\bar\alpha\mid\alpha\in A\}, \qquad
\Lambda_{\max}=\{\alpha\in A\mid \alpha=-\lambda\bar\alpha\}. $$
\noindent
A {\em form parameter} $\Lambda$ is an additive subgroup of $A$ such that
\begin{itemize}
\item[(1)] $\Lambda_{\min}\subseteq\Lambda\subseteq\Lambda_{\max}$,
\smallskip
\item[(2)] $\alpha\Lambda\bar\alpha\subseteq\Lambda$ for all $\alpha\in A$.
\end{itemize}
The pair $(A,\Lambda)$ is called a {\em form ring}.


\subsection{}\label{form ideals}
Let $I\unlhd A$ be a two-sided ideal of $A$. We assume $I$ to be
involution invariant, i.~e.~such that $\bar I=I$. Set
$$ \Gamma_{\max}(I)=I\cap \Lambda, \qquad
\Gamma_{\min}(I)=\{\xi-\lambda\bar\xi\mid\xi\in I\}+
\langle\xi\alpha\bar\xi\mid \xi\in I,\alpha\in\Lambda\rangle. $$
\noindent
A {\em relative form parameter} $\Gamma$ in $(\FormR)$ of level $I$ is an
additive group of $I$ such that
\begin{itemize}
\item[(1)] $\Gamma_{\min}(I)\subseteq \Gamma \subseteq\Gamma_{\max}(I)$,
\smallskip
\item[(2)] $\alpha\Gamma\bar\alpha\subseteq \Gamma$ for all $\alpha\in A$.
\end{itemize}
The pair $(I,\Gamma)$ is called a {\em form ideal}.
\par
In the level calculations we will use sums and products of form
ideals. Let $(I,\Gamma)$ and $(J,\Delta)$ be two form ideals. Their sum
is artlessly defined as $(I+J,\Gamma+\Delta)$, it is immediate to verify
that this is indeed a form ideal.
\par
Guided by analogy, one is tempted to set
$(I,\Gamma)(J,\Delta)=(IJ,\Gamma\Delta)$. However, it is considerably
harder to correctly define the product of two relative form parameters.
The papers \cite{Ha1,Ha2,RH,RH2} introduce the following definition
$$ \Gamma\Delta=\Gamma_{\min}(IJ)+{}^J\Gamma+{}^I\Delta, $$
\noindent
where
$$ {}^J\Gamma=\big\langle \xi\Gamma\bar\xi\mid \xi\in J\big\rangle,\qquad
{}^I\Delta=\big\langle \xi\Delta\bar\xi\mid \xi\in I\big\rangle. $$
\noindent
One can verify that this is indeed a relative form parameter of level $IJ$
if $IJ=JI$.
\par
However, in the present paper we do not wish to impose any such
commutativity assumptions. Thus, we are forced to consider the
symmetrised products
$$ I\circ J=IJ+JI,\qquad
\Gamma\circ\Delta=\Gamma_{\min}(IJ+JI)+{}^J\Gamma+{}^I\Delta\big. $$
\noindent
The notation $\Gamma\circ\Delta$ -- as also $\Gamma\Delta$ is
slightly misleading, since in fact it depends on $I$ and $J$, not
just on $\Gamma$ and $\Delta$. Thus, strictly speaking, one should
speak of the symmetrised products of {\it form ideals\/}
$$ (I,\Gamma)\circ (J,\Delta)=
\big(IJ+JI,\Gamma_{\min}(IJ+JI)+{}^J\Gamma+{}^I\Delta\big). $$
\noindent
Clearly, in the above notation one has
$$ (I,\Gamma)\circ (J,\Delta)=(I,\Gamma)(J,\Delta)+(J,\Delta)(I,\Gamma). $$


\subsection{}\label{quasi-finite}
A {\em form algebra over a commutative ring $R$} is a form ring $(A,\Lambda)$,
where $A$ is an $R$-algebra and the involution leaves $R$ invariant, i.e.,
$\bar R=R$.
\par\smallskip
$\bullet$ A form algebra $(\FormR)$ is called {\it module finite}, if $A$ is
finitely generated as an $R$-module.
\par\smallskip
$\bullet$ A form algebra $(\FormR)$ is called {\it quasi-finite},
if there is a direct system of module finite  $R$-subalgebras $A_i$ of
$A$ such that $\varinjlim A_i=A$.
\par\smallskip
However, in general $\Lambda$ is not an $R$-module. This forces us to
replace $R$ by its subring $R_0$, generated by all $\alpha\bar\alpha$
with $\alpha\in R$. Clearly, all elements in $R_0$ are invariant with
respect to the involution, i.~e.\ $\bar r=r$, for $r\in R_0$.
\par
It is immediate, that any form parameter $\Lambda$ is an $R_0$-module.
This simple fact will be used throughout. This is precisely why we have
to localise in multiplicative subsets of $R_0$, rather than in those of
$R$ itself.


\subsection{}\label{localization}
Let $(A,\Lambda)$ be a form algebra over a commutative ring $R$ with $1$,
and let $S$ be a multiplicative subset of $R_0$, (see \S\ref{quasi-finite}).
For any $R_0$-module $M$ one can consider its localisation $S^{-1}M$
and the corresponding localisation homomorphism $F_S:M\map S^{-1}M$.
By definition of the ring $R_0$ both $A$ and $\Lambda$ are $R_0$-modules,
and thus can be localised in $S$.
\par
In the present paper, we mostly use localisation with respect to the
following two types of multiplication systems of $R_0$.
\par\smallskip
$\bullet$ {\it Principal localisation\/}: for any $s\in R_0$ with $\bar s=s$,
the multiplicative system generated by $s$ is defined as
$\langle s\rangle=\{1,s,s^2,\ldots\}$. The localisation of the form algebra
$(\FormR)$ with respect to multiplicative system $\langle s\rangle$ is usually
denoted by $(A_s,\Lambda_s)$, where as usual $A_s=\langle s\rangle^{-1}A$ and
$\Lambda_s=\langle s\rangle^{-1}\Lambda$ are the usual principal localisations
of the ring $A$ and the form parameter $\Lambda$.
Notice that, for each $\alpha\in A_s$, there exists an integer $n$ and an
element $a\in A$ such that $\displaystyle\alpha=\frac a{s^n}$, and for
each $\xi\in\Lambda_s$, there exists an integer $m$ and an element
$\zeta\in\Lambda$ such that $\displaystyle\xi=\frac\zeta{s^m}$.
\par\smallskip
$\bullet$ {\it Maximal localisation\/}: consider a maximal ideal $\gm\in\Max(R_0)$
of $R_0$ and the multiplicative closed set $S_{\gm}=R_0\backslash\gm$. We
denote the localisation of the form algebra $(\FormR)$ with respect to $S_{\gm}$
by $(A_\gm,\Lambda_\gm)$, where $A_\gm=S_{\gm}^{-1}A$ and
$\Lambda_\gm=S_{\gm}^{-1}\Lambda$ are the usual maximal localisations of the
ring $A$ and the form parameter, respectively.
\par\smallskip
In these cases the corresponding localisation homomorphisms will be
denoted by $F_s$ and by $F_{\gm}$, respectively.
\par
The following fact is verified by a straightforward computation.
\begin{Lem}
For any\/ $s\in R_0$ and for any\/ $\gm\in\Max(R_0)$ the
pairs\/ $(A_s,\Lambda_s)$ and\/ $(A_\gm,\Lambda_\gm)$ are form rings.
\end{Lem}


\section{Unitary groups}\label{sec3}

In the present section we recall basic notation and facts related to
Bak's generalised unitary groups and their elementary subgroups.


\subsection{}\Label{general}
Let, as above, $A$ be an associative ring with 1. For natural $m,n$
we denote by $M(m,n,A)$ the additive group of $m\times n$ matrices
with entries in $A$. In particular $M(m,A)=M(m,m,A)$ is the ring of
matrices of degree $m$ over $A$. For a matrix $x\in M(m,n,A)$ we
denote by $x_{ij}$, $1\le i\le m$, $1\le j\le n$, its entry in the
position $(i,j)$. Let $e$ be the identity matrix and $e_{ij}$,
$1\le i,j\le m$, be a standard matrix unit, i.e.\ the matrix which has
1 in the position $(i,j)$ and zeros elsewhere.
\par
As usual, $\GL(m,A)=M(m,A)^*$ denotes the general linear group
of degree $m$ over $A$. The group $\GL(m,A)$ acts on the free right
$A$-module $V\cong A^{m}$ of rank $m$. Fix a base $e_1,\ldots,e_{m}$
of the module $V$. We may think of elements $v\in V$ as columns with
components in $A$. In particular, $e_i$ is the column whose $i$-th
coordinate is 1, while all other coordinates are zeros.
\par
Actually, in the present paper we are only interested in the case,
when $m=2n$ is even. We usually number the base
as follows: $e_1,\ldots,e_n,e_{-n},\ldots,e_{-1}$. All other
occuring geometric objects will be numbered accordingly. Thus,
we write $v=(v_1,\ldots,v_n,v_{-n},\ldots,v_{-1})^t$, where $v_i\in A$,
for vectors in $V\cong A^{2n}$.
\par
The set of indices will be always ordered accordingly,
$\Omega=\{1,\ldots,n,-n,\ldots,-1\}$. Clearly, $\Omega=\Omega^+\sqcup\Omega^-$,
where $\Omega^+=\{1,\ldots,n\}$ and $\Omega^-=\{-n,\ldots,-1\}$. For an
element $i\in\Omega$ we denote by $\e(i)$ the sign of $\Omega$, i.e.\
$\e(i)=+1$ if $i\in\Omega^+$, and $\e(i)=-1$ if $i\in\Omega^-$.

\subsection{}\Label{unitary} For a form ring $(\FormR)$, one considers the
{\it hyperbolic unitary group\/} $\GU(2n,\FormR)$, see~\cite[\S2]{BV3}.
This group is defined as follows:
\par
One fixes a symmetry $\lambda\in\Cent(A)$, $\lambda\bar\lambda=1$ and
supplies the module $V=A^{2n}$ with the following $\lambda$-hermitian form
$h:V\times V\map A$,
$$ h(u,v)=\bar u_1v_{-1}+\ldots+\bar u_nv_{-n}+
\lambda\bar u_{-n}v_n+\ldots+\lambda\bar u_{-1}v_1. $$
\noindent
and the following $\Lambda$-quadratic form $q:V\map A/\Lambda$,
$$ q(u)=\bar u_1 u_{-1}+\ldots +\bar u_n u_{-n} \mod\Lambda. $$
\noindent
In fact, both forms are engendered by a sesquilinear form $f$,
$$ f(u,v)=\bar u_1 v_{-1}+\ldots +\bar u_n v_{-n}. $$
\noindent
Now, $h=f+\lambda\bar{f}$, where $\bar f(u,v)=\bar{f(v,u)}$, and
$q(v)=f(u,u)\mod\Lambda$.
\par
By definition, the hyperbolic unitary group $\GU(2n,A,\Lambda)$ consists
of all elements from $\GL(V)\cong\GL(2n,A)$ preserving the $\lambda$-hermitian
form $h$ and the $\Lambda$-quadratic form $q$. In other words, $g\in\GL(2n,A)$
belongs to $\GU(2n,A,\Lambda)$ if and only if
$$ h(gu,gv)=h(u,v)\quad\text{and}\quad q(gu)=q(u),\qquad\text{for all}\quad u,v\in V. $$
\par
When the form parameter is not maximal or minimal, these groups are not
algebraic. However, their internal structure is very similar to that
of the usual classical groups. They are also oftentimes called general
quadratic groups, or classical-like groups.

\subsection{}\Label{elementary1}
{\it Elementary unitary transvections\/} $T_{ij}(\xi)$
correspond to the pairs $i,j\in\Omega$ such that $i\neq j$. They come in
two stocks. Namely, if, moreover, $i\neq-j$, then for any $\xi\in A$ we set
$$ T_{ij}(\xi)=e+\xi e_{ij}-\lambda^{(\e(j)-\e(i))/2}\bar\xi e_{-j,-i}. $$
\noindent
These elements are also often called {\it elementary short root unipotents\/}.
\noindent
On the other side for $j=-i$ and $\a\in\lambda^{-(\e(i)+1)/2}\Lambda$ we set
$$ T_{i,-i}(\a)=e+\a e_{i,-i}. $$
\noindent
These elements are also often called {\it elementary long root elements\/}.
\par
Note that $\bar\Lambda=\bar\lambda\Lambda$. In fact, for any element
$\a\in\Lambda$ one has $\bar\a=-\bar\lambda\a$ and thus $\bar\Lambda$ coincides with
the set of products $\bar\lambda\a$, $\a\in\Lambda$. This means that in the
above definition $\a\in\bar\Lambda$ when $i\in\Omega^+$ and $\a\in\Lambda$
when $i\in\Omega^-$.
\par
Subgroups $X_{ij}=\{T_{ij}(\xi)\mid \xi\in A\}$, where $i\neq\pm j$, are
called {\it short root subgroups\/}. Clearly, $X_{ij}=X_{-j,-i}$.
Similarly, subgroups $X_{i,-i}=\{T_{ij}(\a)\mid
\a\in\lambda^{-(\e(i)+1)/2}\Lambda\}$ are called {\it long root subgroups\/}.
\par
The {\it elementary unitary group\/} $\EU(2n,\FormR)$ is generated by
elementary unitary transvections $T_{ij}(\xi)$, $i\neq\pm j$, $\xi\in A$,
and $T_{i,-i}(\a)$, $\a\in\Lambda$, see~\cite[\S3]{BV3}.

\subsection{}\Label{elementary2}
Elementary unitary transvections $T_{ij}(\xi)$ satisfy the following
{\it elementary relations\/}, also known as {\it Steinberg relations\/}.
These relations will be used throughout this paper.
\par\smallskip
(R1) \ $T_{ij}(\xi)=T_{-j,-i}(-\lambda^{(\varepsilon(j)-\varepsilon (i))/2}\bar{\xi})$,
\par\smallskip
(R2) \ $T_{ij}(\xi)T_{ij}(\zeta)=T_{ij}(\xi+\zeta)$,
\par\smallskip
(R3) \ $[T_{ij}(\xi),T_{hk}(\zeta)]=e$, where $h\ne j,-i$ and $k\ne i,-j$,
\par\smallskip
(R4) \ $[T_{ij}(\xi),T_{jh}(\zeta)]=
T_{ih}(\xi\zeta)$, where $i,h\ne\pm j$ and $i\ne\pm h$,
\par\smallskip
(R5) \ $[T_{ij}(\xi),T_{j,-i}(\zeta)]=
T_{i,-i}(\xi\zeta-\lambda^{-\varepsilon(i)}\bar{\zeta}\bar{\xi})$, where $i\ne\pm j$,
\par\smallskip
(R6) \ $[T_{i,-i}(\alpha),T_{-i,j}(\xi)]=
T_{ij}(\alpha\xi)T_{-j,j}(-\lambda^{(\ep(j)-\ep(i))/2}\bar\xi\alpha\xi)$,
where $i\ne\pm j$.
\par\smallskip
Relation (R1) coordinates two natural parametrisations of the same short
root subgroup $X_{ij}=X_{-j,-i}$. Relation (R2) expresses additivity of
the natural parametrisations. All other relations are various instances
of the Chevalley commutator formula. Namely, (R3) corresponds to the
case, where the sum of two roots is not a root, whereas (R4), and (R5)
correspond to the case of two short roots, whose sum is a short root,
and a long root, respectively. Finally, (R6) is the Chevalley commutator
formula for the case of a long root and a short root, whose sum is a root.
Observe that any two long roots are either opposite, or orthogonal, so
that their sum is never a root.


\subsection{}\Label{sub:1.4}
Let $G$ be a group. For any $x,y\in G$, ${}^xy=xyx^{-1}$ and $y^x=x^{-1}yx$
denote the left conjugate and the right conjugate of $y$ by $x$,
respectively. As usual, $[x,y]=xyx^{-1}y^{-1}$ denotes the
left-normed commutator of $x$ and $y$. Throughout the present paper
we repeatedly use the following commutator identities:
\begin{itemize}
\item[(C1)] $[x,yz]=[x,y]\cdot {}^y[x,z]$,
\smallskip
\item[(C$1^+$)]
An easy induction, using identity (C1), shows that
$$\Bigg[x,\prod_{i=1}^k u_i\Bigg]=
\prod_{i=1}^k {}^{\prod_{j=1}^{i-1}u_j}[x,u_{i}], $$
\noindent
where by convention $\prod_{j=1}^0 u_j=1$,
\item[(C2)] $[xy,z]={}^x[y,z]\cdot [x,z]$,
\smallskip
\item[(C$2^+$)]
As in (C$1^+$), we have
$$\Bigg[\prod_{i=1}^k u_i,x\Bigg]=
\prod_{i=1}^k {}^{\prod_{j=1}^{k-i}u_j}[u_{k-i+1},x], $$
\smallskip
\item[(C3)]
${}^{x}\big[[x^{-1},y],z\big]\cdot {}^{z}\big[[z^{-1},x],y\big]\cdot
{}^{y}\big[[y^{-1},z],x\big]=1$,
\smallskip
\item[(C4)] $[x,{}^yz]={}^y[{}^{y^{-1}}x,z]$,
\smallskip
\item[(C5)] $[{}^yx,z]={}^{y}[x,{}^{y^{-1}}z]$,
\smallskip
\item[(C6)] If $H$ and $K$ are subgroups of $G$, then $[H,K]=[K,H]$,
\end{itemize}
Especially important is (C3), the celebrated {\it Hall--Witt
identity\/}. Sometimes it is used in the following form, known as
the {\it three subgroup lemma\/}.
\begin{Lem}{\label{HW1}}
Let\/ $F,H,L\trianglelefteq G$ be three normal subgroups
of\/ $G$. Then
$$ \big[[F,H],L\big]\le \big [[F,L],H\big ]\cdot \big [F,[H,L]\big ]. $$
\end{Lem}


\section{Relative subgroups}\label{sec4}

In this section we recall definitions and basic facts concerning relative
subgroups.

\subsection{}\Label{relative} One associates with a form ideal $(I,\Gamma)$
the following four relative subgroups.
\par\smallskip
$\bullet$ The subgroup $\FU(2n,I,\Gamma)$ generated by elementary unitary
transvections of level $(I,\Gamma)$,
$$ \FU(2n,I,\Ga)=\big\langle T_{ij}(\xi)\mid \
\xi\in I\text{ if }i\neq\pm j\text{ and }
\xi\in\lambda^{-(\epsilon(i)+1)/2}\Gamma\text{ if }i=-j\big\rangle. $$
\par\smallskip
$\bullet$ The {\it relative elementary subgroup\/} $\EU(2n,I,\Gamma)$
of level $(I,\Gamma)$, defined as the normal closure of $\FU(2n,I,\Gamma)$
in $\EU(2n,\FormR)$,
$$ \EU(2n,I,\Ga)={\FU(2n,I,\Ga)}^{\EU(2n,\FormR)}. $$
\par\smallskip
$\bullet$ The {\it principal congruence subgroup\/} $\GU(2n,I,\Ga)$ of level
$(I,\Ga)$ in $\GU(2n,A,\Lambda)$ consists of those $g\in \GU(2n,A,\Lambda)$,
which are congruent to $e$ modulo $I$ and preserve $f(u,u)$ modulo $\Ga$,
$$ f(gu,gu)\in f(u,u)+\Ga, \qquad u\in V. $$
\par\smallskip
$\bullet$ The full congruence subgroup $\CU(2n,I,\Gamma)$ of level
$(I,\Gamma)$, defined as
$$ \CU(2n,I,\Ga)=\big\{ g\in \GU(2n,A,\Lambda) \mid
[g,\GU(2n,A,\Lambda)]\subseteq \GU(2n,I,\Ga)\big\}. $$
\par\smallskip
In some books, including \cite{HO}, the group $\CU(2n,I,\Ga)$
is defined differently. However, in many important situations
these definitions yield the same group.


\subsection{}\Label{relativefacts}
Let us collect several basic facts, concerning relative groups,
which will be used in the sequel. The first one of them asserts that
the relative elementary groups are $\EU(2n,A,\Lambda)$-perfect.
\begin{Lem}\label{hww3}
Suppose either $n\ge 3$ or $n=2$ and $I=\Lambda I+I\Lambda$.
Then
$$ \EU(2n,I,\Gamma)=[\EU(2n,I,\Gamma),\EU(2n,A,\Lambda)]. $$
\end{Lem}
The next lemma gives generators of the relative elementary subgroup
$\EU(2n,I,\Ga)$ as a subgroup. With this end, consider matrices
$$ Z_{ij}(\xi,\zeta)={}^{T_{ji}(\zeta)}T_{ij}(\xi)
=T_{ji}(\zeta)T_{ij}(\xi)T_{ji}(-\zeta), $$
\noindent
where $\xi\in I$, $\zeta\in A$, if $i\neq\pm j$, and
$\xi\in\lambda^{-(\e(i)+1)/2}\Gamma$,
$\zeta\in\lambda^{-(\e(j)+1)/2}\Lambda$, if $i=-j$.
The following result is \cite{BV3}, Proposition 5.1.
\begin{Lem}\label{genelm}
Suppose $n\ge 3$. Then
\begin{multline*}
\EU(2n,I,\Ga)=\big\langle Z_{ij}(\xi,\zeta)\mid \
\xi\in I,\zeta\in A\text{ if }i\neq\pm j\text{ and }\\
\xi\in\lambda^{-(\epsilon(i)+1)/2}\Gamma,
\zeta\in\lambda^{-(\epsilon(j)+1)/2}\Lambda,
\text{ if }i=-j\big\rangle.
\end{multline*}
\end{Lem}
The following lemma was first established in~\cite{B1}, but remained
unpublished. See~\cite{HO} and~\cite{BV3}, Lemma 4.4, for published
proofs.
\begin{Lem}
The groups $\GU(2n,I,\Gamma)$ and $\CU(2n,I,\Gamma)$ are normal in
$\GU(2n,A,\Lambda)$.
\end{Lem}

Also, throughout the paper we use the absolute and the relative
standard commutator formulas, which were already stated in the
introduction as Theorem~\ref{main_BV} and Theorem~\ref{main_HVZ}.


\subsection{}\Label{sub:1.3}
The proofs in the present paper critically depend on the fact that
the functors $\GU_{2n}$ and $\EU_{2n}$ commute with direct limits.
This idea is used twice.
\par\smallskip
$\bullet$ Analysis of the quasi-finite case can be reduced to the case,
where $A$ is module finite over $R_0$, whereas $R_0$ itself is Noetherian.
Indeed, if $(\FormR)$ is quasi-finite, (see \S\ref{quasi-finite}),
it is a direct limit $\varinjlim\big((A_j)_{R_j},\Lambda_j\big)$ of
an inductive system of form sub-algebras
$\big((A_j)_{R_j},\Lambda_j\big)\subseteq(A_R,\Lambda)$
such that each $A_j$ is module finite over $R_j$, $R_0\subseteq R_j$ and $R_j$
is finitely generated as an $R_0$-module. It follows that $A_j$ is finitely
generated as an $R_0$-module, see \cite[Cor.~3.8]{RH}. This reduction to module
finite algebras will be used in Lemma~\ref{Lem:08} and Theorem~\ref{main_triple}.
\par\smallskip
$\bullet$ Analysis of any localisation can be reduced to the case of principal
localisations. Indeed, let $S$ be a multiplicative system in a commutative
ring $R$. Then $R_s$, $s\in S$, is an inductive system with respect to the
localisation maps $F_t:R_s\to R_{st}$. Thus, for any functor $\mathcal F$
commuting with direct limits one has
${\mathcal F}(S^{-1}R)=\varinjlim{\mathcal F}(R_s)$.
\par\smallskip
The following crucial lemma relies on both of these reductions. In fact,
starting from the next section, we will be mostly working in the principal
localisation $A_t$. However, eventually we shall have to return to the algebra
$A$ itself. In general, localisation homomorphism $F_S$ is not injective,
so we cannot pull elements of $\GU(2n,S^{-1}A,S^{-1}\Lambda)$ back to
$\GU(2n,A,\Lambda)$.
However, over a {\it Noetherian} ring, {\it principal\/} localisation homomorphisms
$F_t$ are indeed injective on small $t$-adic neighbourhoods of identity!

\begin{Lem}\Label{Lem:03}
Let $R$ be a commutative Noetherian ring and let $A$ be a module finite
$R$-algebra. Then for any $t\in R$ there exists a positive integer $l$
such that restriction
$$ F_t:\GU(2n,t^lA,t^l\Lambda)\to\GU(2n,A_t,\Lambda_t), $$
\noindent
of the localisation map to the principal congruence subgroup of level
$(t^lA,t^l\Lambda)$ is injective.
\end{Lem}
\begin{proof} The proof follows from the injectivity of the localisation map
$F_t:t^lA\rightarrow A_t$, see \cite[Lem\-ma~4.10]{B4}.
\end{proof}


\section{Levels of mixed commutator subgroups}\label{sec5}

In the present section we closely follow the notation and computations
of \cite{RNZ1}. For the proof of Theorem~\ref{main_triple}  it is absolutely vital to
improve the level calculations from \cite{RNZ1}, \S\ref{sec8}. Specifically,
here we amalgamate Lemmas~22 and 23 therefrom, and streamline their
proofs. As before, we assume that $(A,\Lambda)$ is a form ring over
a commutative ring $R$ with involution, $R_0$ is the subring of $R$,
generated by $a\bar a$, where $a\in R$, as in \S\ref{quasi-finite},
$(I,\Gamma)$ and $(J,\Delta)$ are two form ideals of $(\FormR)$ and,
finally, $n\ge 3$. In this setting, in \S\ref{sec2} we have defined the
symmetrised product of form parameters $\Gamma$ and $\Delta$ as
$$ \Gamma\circ\Delta={}^J\Gamma+{}^I\Delta+\Gamma_{\min}(IJ+JI), $$
\noindent
which is a relative form parameter of level $I\circ J=IJ+JI$. Notice
that for any form ideal $(I,\Gamma)$ of the form ring $(\FormR)$,
we have $\Gamma=\Gamma\circ\Lambda$.

\par
First, we recall the rough level calculation of mixed commutator
subgroups, which was essentially contained already in
\cite{Ha1, Ha2, RH, RH2} and reproduced in more details in \cite{RNZ1}.
The left inclusion in the following lemma is \cite{RNZ1}, Lemma 21,
while the right inclusion is \cite{RNZ1}, Lemma 23.

\begin{Lem}\Label{Lem:Habdank}
Let $(I,\Gamma)$ and $(J,\Delta)$ be two form ideals of a form ring
$(\FormR)$. Then
\begin{multline*}
\EU(2n,(I,\Gamma)\circ(J,\Delta))
\subseteq \big [\FU(2n,I,\Gamma),\FU(2n,J,\Delta)\big]\\
\subseteq \big[\EU(2n,I,\Gamma),\EU(2n,J,\Delta)\big]
\subseteq \GU(2n,(I,\Gamma)\circ(J,\Delta)).
\end{multline*}
\end{Lem}
Actually, \cite{RNZ1}, Lemma 23, asserted a bit more, namely that
one occurrence of the relative elementary subgroup can be
replaced by the corresponding principal congruence subgroup,
$$ \big[\EU(2n,I,\Gamma),\GU(2n,J,\Delta)\big]
\subseteq \GU(2n,(I,\Gamma)\circ(J,\Delta)). $$
\noindent
Does this inclusion hold when {\it both\/} relative elementary
subgroup are replaced by the corresponding principal congruence
subgroups? For all $n\ge 2$ Lemma 22 of \cite{RNZ1}, established
a similar, but weaker inclusion
$$ \big[\GU(2n,I,\Gamma),\GU(2n,J,\Delta)\big]
\subseteq \GU(2n,I\circ J,\Gamma_{\max}(I\circ J)), $$
\noindent
with the {\it maximal\/} relative form parameter of level $I\circ J$
on the right hand side, instead of the symmetrised product of the
relative form parameters. For $n\ge 3$ we can in fact merge these
results, but the argument is not straightforward, this is why
we missed it when writing \cite{RNZ1}. This argument refers
to the structure theorems for the {\it stable\/} unitary groups
established in \cite{B1,bass73}, see also \cite{ZZ,RN}.

\begin{Lem}\label{Lem:GGG}
Let $(\FormR)$ be a form ring and $(I,\Gamma)$ and  $(J,\Delta)$
be form ideals of  $(\FormR)$.  Then we have
\begin{equation}\label{ohdmre}
\big[\GU(2n,I,\Gamma),\GU(2n,J,\Delta)\big]\subseteq
\GU(2n,(I,\Gamma)\circ(J,\Delta)).
\end{equation}
\end{Lem}
\begin{proof}
We first show that (\ref{ohdmre}) holds for the {\it stable\/}
unitary groups, namely that
\begin{equation}\label{stable1}
[\GU(I,\Gamma),\GU(J,\Delta)]\subseteq\GU((I,\Gamma)\circ(J,\Delta)).
\end{equation}
\noindent
By the stable analogue of Lemma~\ref{Lem:Habdank}, which immediately
follows by passage to limits, we have inclusions
\begin{equation}\label{stable2}
        \EU((I,\Gamma)\circ(J,\Delta))\subseteq
[\EU(I,\Gamma),\EU(J,\Delta)]\subseteq[\GU(I,\Gamma),\GU(J,\Delta)]
\end{equation}
and
\begin{equation}\label{stable3}
[\EU(I,\Gamma),\EU(J,\Delta)]\subseteq \GU((I,\Gamma)\circ(J,\Delta)).
\end{equation}

Since the subgroup $[\GU(I,\Gamma),\GU(J,\Delta)]$ is normalized by
$E(\FormR)$, applying Bass' sandwich theorem, see~\cite[Theorem~5.4.10]{HO},
we can conclude that there exists a unique form ideal $(K,\Omega)$ such
that
\begin{equation}\label{stable4}
\EU(K,\Omega)\subseteq [\GU(I,\Gamma),\GU(J,\Delta)]\subseteq\GU(K,\Omega).
\end{equation}
By Lemma~\ref{HW1}, we get
\begin{multline*}
\big[[\GU(I,\Gamma),\GU(J,\Delta)],\EU(\FormR)\big]\subseteq\\
\big[[\GU(I,\Gamma),\EU(\FormR)],\GU(J,\Delta)\big]\cdot
\big[[\GU(J,\Delta),\EU(\FormR)],\GU(I,\Gamma)\big].
\end{multline*}
\noindent
But the absolute commutator formula implies that
\begin{multline}
\big[[\GU(I,\Gamma),\EU(\FormR)],\GU(J,\Delta)\big]\cdot
\big[[\GU(J,\Delta),\EU(\FormR)],\GU(I,\Gamma)\big]=\\
[\EU(I,\Gamma),\EU(J,\Delta)].
\end{multline}
\noindent
Thus,
\begin{equation}\label{pjhenis}
\big[[\GU(I,\Gamma),\GU(J,\Delta)],\EU(\FormR)\big]
\subseteq[\EU(I,\Gamma),\EU(J,\Delta)].
\end{equation}
\noindent
Again by the general commutator formula and (\ref{stable3}), we have
\begin{multline}\label{assme}
\EU((I,\Gamma)\circ(J,\Delta))=
[\EU((I,\Gamma)\circ(J,\Delta)),\EU(\FormR)]\\
\subseteq\big[[\EU(I,\Gamma),\EU(J,\Delta)],\EU(\FormR)\big]\\
\subseteq[\GU((I,\Gamma)\circ(J,\Delta)),\EU(\FormR)]
=\EU((I,\Gamma)\circ(J,\Delta)).
\end{multline}
\noindent
Forming another commutator of~(\ref{pjhenis}) with $\EU(\FormR)$
and applying the inequalities obtained in~(\ref{assme}) we get
$$ \Big[\big[[\GU(I,\Gamma),\GU(J,\Delta)],\EU(\FormR)\big],
\EU(\FormR)\Big]=\EU((I,\Gamma)\circ(J,\Delta)). $$
\noindent
Using inclusions (\ref{stable4}), we see that
\begin{multline*}
\EU(K,\Omega)=\big[[\EU(K,\Omega),\EU(\FormR)],\EU(\FormR)\big]\\
\subseteq \Big[\big[[\GU(I,\Gamma),\GU(J,\Delta)],\EU(\FormR)\big],
\EU(\FormR)\Big]
=\EU((I,\Gamma)\circ(J,\Delta))\\
=\big[[\EU((I,\Gamma)\circ(J,\Delta)),\EU(\FormR)],\EU(\FormR)\big]\\
\subseteq \Big[\big[[\GU(I,\Gamma),\GU(J,\Delta)],\EU(\FormR)\big],
\EU(\FormR)\Big]\\
\subseteq \big[[\GU(K,\Omega),\EU(\FormR)],\EU(\FormR)\big]
=\EU(K,\Omega).
\end{multline*}
\noindent
Thus, we can conclude that $\EU(K,\Omega)=\EU((I,\Gamma)\circ(J,\Delta))$.
This implies that $(K,\Omega)=(I,\Gamma)\circ(J,\Delta)$, see the
second paragraph of the proof of~\cite[Theorem~5.4.10]{HO}.
Substituting this equality in (\ref{stable4}), we see that inclusion
(\ref{stable1}) holds at the stable level, as claimed.
\par
Let $\varphi$ denote the usual stability embedding
$\varphi:\GU(2n,\FormR)\to\GU(\FormR)$. Then
\begin{multline*}
\varphi\big(\big[\GU(2n,I,\Gamma),\GU(2n,J,\Delta)\big]\big)=
\big[\varphi\big(\GU(2n,I,\Gamma)\big),
\varphi\big(\GU(2n,J,\Delta)\big)\big]\subset\\
\big[\GU(I,\Gamma),\GU(J,\Delta)\big].
\end{multline*}
\noindent
In particular, the result at the stable level implies that
$$ \varphi\big(\big[\GU(2n,I,\Gamma),\GU(2n,J,\Delta)\big]\big)
\subseteq\varphi\big(\GU(2n,\FormR)\big)\cap
\GU((I,\Gamma)\circ(J,\Delta)). $$
\noindent
On the other hand,
$$ \varphi\big(\GU(2n,\FormR)\big)\cap\GU((I,\Gamma)\circ(J,\Delta))=
\varphi\big(\GU(2n,(I,\Gamma)\circ(J,\Delta))\big). $$
\noindent
Since $\varphi$ is injective, we can conclude that
$$ [\GU(2n,I,\Gamma),\GU(2n,J,\Delta)]\subseteq
\GU(2n,(I,\Gamma)\circ(J,\Delta)), $$
\noindent
as claimed.
\end{proof}


\section{Generation of mixed commutator subgroups}\label{sec6}

Our next result is a higher analogue of Lemma~\ref{genelm}. Despite its
technical character, it is one of the main new tools of our
proof. Actually, for applications to width problems it is more
expedient to construct a shorter set of generators, and this will
be done in our subsequent paper. The form below is especially
adjusted for the version of the relative commutator calculus
we cultivate in the two following sections.

\begin{The}\Label{Comgenerator}
Let $(\FormR)$ be a form ring and $(I,\Gamma)$, $(J,\Delta)$ be two
form ideals of $(\FormR)$. Then the mixed commutator subgroup
$\big[\EU(2n,I,\Gamma),\EU(2n,J,\Delta)\big]$ is generated as
a group by the elements of the form
\par\smallskip
$\bullet$ ${}^{c}\big[T_{ji}(\alpha),{}^{T_{ij}(a)}T_{ji}(\beta)\big]$,
\par\smallskip
$\bullet$ ${}^c\big[T_{ji}(\alpha),T_{ij}(\beta)\big]$,
\par\smallskip
$\bullet$ ${}^c T_{ij}(\xi)$,
\par\smallskip\noindent
where $T_{ji}(\alpha)\in\EU(2n,I,\Gamma)$,
$T_{ji}(\beta)\in \EU(2n,J,\Delta)$,
$T_{ij}(\xi)\in\EU\big(2n,(I,\Gamma)\circ(J,\Delta)\big)$,
and $T_{ij}(a),c\in\EU(2n,\FormR)$.
\end{The}
\begin{proof}
A typical generator of $\big[\EU(2n,I,\Gamma),\EU(2n,J,\Delta)\big]$
is of the form $[e,f]$, where $e\in\EU(2n,I,\Gamma)$ and
$f\in\EU(2n,J,\Delta)$. Thanks to Lemma~\ref{genelm}, we may assume
that $e$ and $f$ are products of elements of the form
$$ e_{i}=Z_{pq}(a,\alpha),\qquad f_{j}=Z_{rs}(b,\beta), $$
\noindent
where $a\in A$ and $\alpha\in I$ if $p\ne\pm q$ and
$a\in\lambda^{-(\ep(p)+1)/2}\Lambda$,
$\alpha\in\lambda^{-(\ep(-p)+1)/2}\Gamma$ if $p=-q$, and
where $b\in A$ and $ \beta \in J$ if $r\ne \pm s$ and
$b\in \lambda^{-(\ep(r)+1)/2}\Lambda$,
$\beta\in \lambda^{-(\ep(-r)+1)/2}\Delta $ if $r=-s$.
\par
Applying (C$1^+$) and then (C$2^+$), one gets that
$\big[\EU(2n,I,\Gamma),\EU(2n,J,\Delta)\big]$ is generated
by the elements of the form
$$ ^{c}\big[^{T_{pq}(a)}T_{qp}(\alpha),{}^{T_{ij}(b)}T_{ji}(\beta)\big], $$
\noindent
where $c\in\EU(2n,\FormR)$. Furthermore,
$$ {}^c\big[Z_{pq}(a,\alpha),{}^{T_{ij}(b)}T_{ji}(\beta)\big]=
{}^{cT_{pq}(a)}\big[T_{qp}(\alpha),
{}^{T_{pq}(-a)T_{ij}(b)}T_{ji}(\beta)\big]. $$
\noindent
The normality of $\EU(2n,J,\Delta)$ implies that
${}^{T_{pq}(-a)T_{ij}(b)}T_{ji}(\beta)\in\EU(2n,J,\Delta)$,
which is a product of  $Z_{ij}(\xi,\zeta)$ by Lemma~\ref{genelm}.
Again by (C$1^+$), one reduces the proof to the case
of showing that
$$ \big[T_{pq}(\alpha),Z_{ij}(b,\beta)\big] $$
\noindent
is a product of the generators listed above. We divide
the proof into 3 cases, namely,
\begin{itemize}
\item[I.] $T_{pq}(\alpha)$ and $Z_{ij}(b,\beta)$ have
opposite indices, namely $p=j $ and $q=i$.
\smallskip
\item[II.] $T_{pq}(\alpha)$ and $Z_{ij}(b,\beta)$ have the same indices.
\smallskip
\item[III.] Otherwise.
\end{itemize}
\par\smallskip

Case I,  $T_{pq}(\alpha)$ and $Z_{ij}(b,\beta)$ are opposite,
then $p=j $ and $q=i$.
$$ \big[T_{ji}(\alpha),Z_{ij}(b,\beta)\big]=
\big[T_{ji}(\alpha),{}^{T_{ji}(b)}T_{ij}(\beta)\big]=
{}^{T_{ji}(b)}\big[T_{ji}(\alpha),T_{ij}(\beta)\big], $$
\noindent
which is a generator listed in the current theorem.
\par\smallskip
Case II, there is nothing to prove.
\par\smallskip
Case III, the proof can be further subdivided into the
following subcases:
\begin{itemize}
\item[1.] $T_{pq}(\alpha)$ commutes with $Z_{ij}(b,\beta)$.
$\big[T_{pq}(\alpha),Z_{ij}(b,\beta)\big]=1$
which satisfies the lemma.
\item[2.] $T_{pq}(\alpha)$ and $Z_{ij}(b,\beta)$ are short
roots, $q=i$ and  $p\ne \pm j$. Using (C4) and (R3), one gets
\begin{eqnarray*}
\big[T_{pi}(\alpha),Z_{ij}(b,\beta)\big]&=&
\big[T_{pi}(\alpha),{}^{T_{ji}(b)}T_{ij}(\beta)\big]\\
&=&{}^{T_{ji}(b)}\big[T_{pi}(\alpha)^{T_{ji}(b)},T_{ij}(\beta)\big]\\
&=&{}^{T_{ji}(b)}\big[T_{pi}(\alpha), T_{ij}(\beta)\big]\\
&=&{}^{T_{ji}(b)}T_{pj}(\alpha\beta),
\end{eqnarray*}
\noindent
which is a generator of $\big[\EU(2n,I,\Gamma),\EU(2n,J,\Delta)\big]$.
\item[3.] $T_{pq}(\alpha)$ and $Z_{ij}(b,\beta)$ are short roots,
$q=i$ and  $p=- j$. Using (C4) and (R5), one gets
\begin{eqnarray*}
\big[T_{-j,i}(\alpha),Z_{ij}(b,\beta)\big]&=&
\big[T_{-j,i}(\alpha),{}^{T_{ji}(b)}T_{ij}(\beta)\big]\\
&=&{}^{T_{ji}(b)}\big[T_{-j,i}(\alpha)^{T_{ji}(b)},T_{ij}(\beta)\big]\\
&=&{}^{T_{ji}(b)}\big[T_{-j,i}(\alpha), T_{ij}(\beta)\big]\\
&=&{}^{T_{ji}(b)}T_{-j,j}(\alpha\beta-\lambda^{\ep(j)}\bar\beta\bar\alpha),
\end{eqnarray*}
which is of the form in the theorem.

\item[4.] $T_{pq}(\alpha)$ and $Z_{ij}(b,\beta)$ are short roots,
$p=j$ and $q\ne\pm i$. Using  (R1), we reduce our consideration
to the subcase (2).

\item[5.] $T_{pq}(\alpha)$ and $Z_{ij}(b,\beta)$ are short roots,
$p=j$ and $q\ne- i$. Using  (R1), we reduce our consideration to the subcase (3).

\item[6.] $T_{pq}(\alpha)$ is a long root and $Z_{ij}(b,\beta)$
is a short root,  $q=i$. Using (R6), we get
\begin{eqnarray*}
\big[T_{-i,i}(\alpha), Z_{ij}(b,\beta)\big]&=&\big[T_{-i,i}(\alpha),
{}^{T_{ji}(b)}T_{ij}(\beta)\big]\\
&=&{}^{T_{ji}(b)}\big[T_{-i,i}(\alpha)^{T_{ji}(b)},T_{ij}(\beta)\big]\\
&=&{}^{T_{ji}(b)}\big[T_{-i,i}(\alpha), T_{ij}(\beta)\big]\\
&=&{}^{T_{ji}(b)}\big(T_{-i,j}(\alpha\beta)
T_{-j,j}(-\lambda^{(\ep(j)+\ep(i))/2}\beta\alpha\bar\beta)\big)\\
&=&\Big({}^{T_{ji}(b)}T_{-i,j}(\alpha\beta)\Big)
\Big({}^{T_{ji}(b)}T_{-j,j}(-\lambda^{(\ep(j)-\ep(i))/2}\beta\alpha\bar\beta)\Big),
\end{eqnarray*}
\noindent
which is a product of the generators in the lemma.

\item[7.] $T_{pq}(\alpha)$ is a long root and $Z_{ij}(b,\beta)$
is a short root,  $p=j$. Using  (R1), we reduce our consideration
to the subcase (6).

\item[8.] $T_{pq}(\alpha)$ is a short root and $Z_{ij}(b,\beta)$
is a long root, $q=i$. Using (R6), we have
\begin{eqnarray*}
\big[T_{pi}(\alpha), Z_{i,-i}(b,\beta)\big]&=&\big[T_{pi}(\alpha),
{}^{T_{-i,i}(b)}T_{i,-i}(\beta)\big]\\
&=&{}^{T_{-i,i}(b)}\big[T_{pi}(\alpha)^{T_{-i,i}(b)},T_{i,-i}(\beta)\big]\\
&=&{}^{T_{-i,i}(b)}\big[T_{pi}(\alpha), T_{i,-i}(\beta)\big]\\
&=&{}^{T_{-i,i}(b)}\big(T_{p,-i}(\alpha\beta)
T_{p,-p}(-\lambda^{(\ep(i)-\ep(p))/2}\alpha\beta\bar\alpha)\big)\\
&=&\Big({}^{T_{-i,i}(b)}T_{p,-i}(\alpha\beta)\Big)
\Big({}^{T_{-i,i}(b)}T_{p,-p}(-\lambda^{(\ep(i)-\ep(p))/2}\alpha\beta\bar\alpha)\Big),
\end{eqnarray*}
\noindent
which is a product of the generators in the lemma.

\item[9.] $T_{pq}(\alpha)$ is a short root and $Z_{ij}(b,\beta)$
is a long root, $p=j$. Using (R1), we reduce it to the subcase (8).
This finishes the proof of Case III, hence the whole proof. \qedhere
\end{itemize}
\end{proof}


\section{Commutator calculus}\label{sec7}

This and the next two sections constitute the technical heart of
the paper.
\par
Let us recall some facts from \cite{RNZ1}. For any $t\neq 0\in R_0$
and any given positive integer $l$, the
set $t^lA$ is in fact an ideal of the algebra $A$. Similarly, it is
straightforward to verify that
$t^l\Lambda=\{t^l\alpha\mid \alpha\in\Lambda\}$ is in fact a relative
form parameter for $t^lA$, and, thus, $(t^lA,t^l\Lambda)$
is a form ideal. This allows us to define  the corresponding
groups $\FU(2n,t^lA,t^l\Lambda)$ and $\FU(2n,t^lI,t^l\Gamma)$. To
make calculations somewhat less painful, we introduce the group
$\FUnT{l}{I}$ which by definition is the normal closure of
$\FUnt{l}{I}$ in $\FFUnt{l}$,
$$ \FUnT{l}{I}={}^{\FFUnt{l}}{\FUnt{l}{I}}\unlhd\FFUnt{l}. $$
\noindent
Actually, the use of this base of $t$-adic neighbourhoods instead of
the usual ones is precisely one of the key technical tricks of
\cite{RHZZ1,RNZ1,RNZ2}. Normality of $\FUnT{l}{I}$ in $\FFUnt{l}$
will be repeatedly used in the sequel. Observe, that
$\EUnT{l}{A}=\EUnt{l}{A}$.
\par
Let us introduce a further piece of notation. For a form ideal
$(I,\Gamma)$ and an element $t\in R_0$, the set
$\FU^1\Big(2n,\frac{I}{t^m},\frac{\Gamma}{t^m}\Big)$ consists
of elementary unitary transvections $T_{ij}(a)$, such that
$a\in\frac{I}{t^m}$ if $i\neq\pm j$ and
$a\in\lambda^{\ep(i)+1)/2}\frac{\Gamma}{t^m}$ if $i=-j$. Denote by
$\FU^L\Big(2n,\frac{A}{t^m},\frac{I}{t^m},\frac{\Gamma}{t^m}\Big)$
the set of products of $\le L$ elements of the form
${}^{\FU^1\Big(2n,\frac{A}{t^m},\frac{\Lambda}{t^m}\Big)}
\FU^1\Big(2n,\frac{I}{t^m},\frac{\Gamma}{t^m}\Big)$.
The set $\FU^1(2n,t^mI,t^m\Gamma)$ is defined similarly. By the
same token, $\FU^K(2n,t^mI,t^m\Gamma)$, denotes the set of products
of $\le K$ elements of $\FU^1(2n,t^mI,t^m\Gamma)$.

The next result is a summary of the relative conjugation calculus
and relative commutator calculus, as developed in \cite{RNZ1},
Lemmas~8, 11 and 12.

\begin{Lem}\Label{Lem:cong2}
For any given $m,l$ there exists a sufficiently large integer $p$
such that
\begin{equation}\Label{eqn:RVZ1}
{}^{\FU^1\big(2n,\frac{A}{t^m},\frac{\Lambda}{t^m}\big)}\FUnt{p}{I}\le\FUnT{l}{I},
\end{equation}
\noindent
there exists an integer $p$ such that
\begin{align}
{}^{\FU^1\big(2n,\frac{A}{t^m},\frac{\Lambda}{t^m}\big)}
\big[\FUnT{p}{I}, &\FU(2n,t^pA,t^pJ,t^p\Delta)\big]\notag\\
&\subseteq\big[\FUnT{l}{I},\FU(2n,t^lA,t^lJ,t^l\Delta)\big],\Label{eqn:RVZ2}
\end{align}
and there exists an integer $p$ such that
\begin{align}
\Big[\FUnTk{p}{I}{},\FU^1\Big(2n,\frac{J}{t^m},&\frac{\Delta}{t^m}\Big)\Big]\notag \\
&\subseteq\big[\FUnT{l}{I},\FUnTJ{l}{J}\big].\Label{eqn:RVZ3}
\end{align}
\end{Lem}

With the use of the Hall--Witt identity (C3) this lemma
immediately implies the following result.

\begin{Lem}\Label{Lem:RVZ4}
For any given $m,l,L$ there exists a sufficiently large integer
$p$ such that
\begin{multline}
\Big[\FUnTk{p}{I}{},\FU^L\Big(2n,\frac{A}{t^m},\frac{J}{t^m},
\frac{\Delta}{t^m}\Big)\Big]\notag \\
\subseteq\big[\FUnT{l}{I},\FUnTJ{l}{J}\big].\Label{eqn:RVZ4}
\end{multline}
\end{Lem}

In the following three Lemmas, as in Lemma~\ref{Lem:cong2}, all
the calculations take place in the fraction ring $(A_t,\Lambda_t)$
(see \S\ref{localization}). All the subgroups of $\GU(2n,A_t,\Lambda_t)$
used in the Lemmas, such as $\EU(2n, A,I, \Gamma)$ or $\GU(2n,J,\Delta)$
are in fact the homomorphic images of the similar subgroups in
$\GU(2n,A_t,\Lambda_t)$ under the natural homomorphism $A\rightarrow A_t$.
Since  lemmas such as Lemma~\ref{Lem:Habdank} and the generalized
commutator formula (Theorem~\ref{main_HVZ}) hold for these subgroups
in $\GU(2n,A,\Lambda)$, they also hold for their corresponding homomorphic
images in $\GU(2n,A_t,\Lambda_t)$.
\begin{Lem}\Label{Lem:EGG}
Let $(\FormR)$ be a form algebra, $(I,\Gamma)$ and $(J,\Delta)$ form
ideals of $(\FormR)$, $m$ an integer and $t\in R_0$. If
$x\in\FU^1(2n,\frac{I}{t^m},\frac{\Gamma}{t^m})$, $l$ is a given
integer, then for every integer $p\ge l+m$  we have
$$ [x,h]\in\GU(2n,t^l(JI+IJ),t^l(\Gamma\circ\Delta)), $$
\noindent
where $h\in\GU(2n,t^pJ,t^p\Delta)$.
\end{Lem}
\begin{proof}
Suppose that $x=T_{sk}(\a)$, $\a\in \frac{I}{t^m}$ for $s\ne-k$,
and $\a\in \lambda^{-(\e(s)+1)/2}\frac{\Gamma}{t^m}$ for $s=-k$. Let
\begin{eqnarray*}
g=[x,h].
\end{eqnarray*}
By a similar argument as used in  Lemma~13 in \cite{RHZZ2}, it
suffices to verify that there exists an integer $p$, such that
$$ \displaystyle\sum_{1\le i\le n}\bar g_{ij}g_{-i,j}\in
t^l(\Gamma\circ\Delta) $$
\noindent
for any given $j$ with $-n\le j\le n$. We divide the proof into 2
cases according to the type of $T_{sk}(\a)$, namely long or short
root. We provide a detailed calculation for the case of a long root
type element. The case of a short root type element is settled by a
similar calculation which will be omitted.
\par\smallskip
Case I. If $T_{sk}(\a)$ is a long root, i.e., $s=-k$ and
$\a\in\lambda^{-(\e(s)+1)/2}\frac{\Gamma}{t^m}$, then
\begin{eqnarray*}
g=[T_{s,-s}(\a),h]=T_{s,-s}(\a)\Bigl(e-\sum_{i,j}h_{i,s}\a\bar h_{-j,s}\Bigr),
\end{eqnarray*}
where $h_{ij} \in t^pI$.
\par
Let us have a closer look at the sum
$\displaystyle\sum_{1\le i\le n}\bar g_{ij}g_{-i,j}$.
When $j\ne-s$, we may, without loss of generality, assume that $s\ge 0$
and $j\ge 0$, and thus this sum can be rewritten in the form
\begin{multline*}
\sum_{1\le i\le n}\overline{ h_{i,s}\a\bar h_{-j,s}}h_{-i,s}\a\bar h_{-j,s}
-\lambda^{(\e(j)-\e(-s))/2} h_{-j,s}\a\bar h_{-j,s}+
\overline { \a h_{-s,s}\a\bar h_{-j,s}}h_{-s,s}\a\bar h_{-j,s}  \\
=\sum_{1\le i\le n}h_{-j,s}\lambda\bar\a
\bar h_{i,s} h_{-i,s}\a\bar h_{-j,s} -
h_{-j,s}\bar\lambda\a\bar h_{-j,s}+
h_{-j,s}\bar\a \bar h_{-s,s}\bar\a h_{-s,s}\a\bar h_{-j,s},
\end{multline*}
where the first summand belongs to $t^{4p-2m}({}^{I}\Delta)$, whereas
the second and the third ones belong to $t^{2p-m}({}^{J}\Gamma)$ and
$t^{4p-3m}({}^{J}\Gamma)$, respectively.

On the other hand, when $j=-s$, this sum equals
$$ \sum_{1\le i\le n}\overline{h_{is}\a\bar h_{ss}}h_{-i,s}\a\bar h_{ss}
-h_{ss}\bar\a\bar h_{ss}+\big(\bar\a-\overline{\a h_{-s,s}\a\bar h_{ss}})
(1-h_{-s,s}\a\bar h_{ss}\big), $$
where the first sum belongs to $t^{4p-2m}({}^{I}\Delta)$, while the rest equals
\begin{multline*}
x=-h_{ss}\bar\a\bar h_{ss}+(\bar\a-{h_{ss}\bar\a\bar
h_{-s,s}\bar\a})(1-h_{-s,s}\a\bar h_{ss})= \\
-h_{ss}\bar\a\bar h_{ss}+\bar\a-\bar\a h_{-s,s}\a\bar h_{ss}
-{h_{ss}\bar\a\bar h_{-s,s}\bar\a}+{ h_{ss}\bar\a\bar h_{-s,s}\bar\a}
h_{-s,s}\a\bar h_{ss}=\\
-(1+h_{ss}-1)\bar\a(1+\overline{h_{ss}-1})+\bar\a +
\Big(\lambda\bar\a h_{-s,s}\bar\a\bar h_{ss}-
{h_{ss}\bar\a\bar h_{-s,s}\bar\a}\Big)+
h_{ss}\bar\a\bar h_{-s,s}\bar\a h_{-s,s}\a\bar h_{ss}
\end{multline*}
where the two last summands belong to $t^{2p-2m}\Gamma_{\min}((IJ+JI))$ and to
$t^{4p-3m}({}^{J}\Gamma)$, respectively. Thus, for the left summands, one has
\begin{align*}
 -(1+h_{ss}-1)\bar\a(1&+\overline{h_{ss}-1})+\bar\a\\&
=-(h_{ss}-1)\bar\a+\lambda\a\overline{(h_{ss}-1)}-
(h_{ss}-1)\a\overline{(h_{ss}-1)}, \end{align*}
where the first summand also belongs to $t^{p-m}\Gamma_{\min}((IJ+JI))$,
whereas the second one belongs to $t^{2p-m}({}^{J}\Gamma)$, respectively.

Now by our assumption $p\ge l+m$, in both cases the desired sum belongs to
$t^l(\Gamma\circ\Delta)$, as claimed.
\end{proof}

\begin{Lem}\Label{Lem:FFEF}
Let $(\FormR)$ be a module finite form algebra, $(I,\Gamma)$, $(J,\Delta)$
and $(K,\Omega)$  form ideals of  $(\FormR)$, and   $t\in R_0$.  For
any given  integers $m, l$, there is a  sufficiently large integer $p$,
such that
\begin{multline}
\Big[\big[\FU(2n,t^pA,t^pI,t^p\Gamma),
\FU(2n,t^pA,t^pJ,t^p\Delta)\big],\big[\EU(2n,\FormR),
\FU^1(2n,\frac{K}{t^{m}},\frac{\Omega}{t^{m}})\big] \Big]\notag\\
\subseteq\Big[\big[\FU(2n,t^lA ,t^lI,t^l\Gamma),
\FU(2n,t^lA ,t^lJ,t^l\Delta)\big], \FU(2n,t^lA, t^lK,t^l\Omega)\Big].
\end{multline}
\end{Lem}
\begin{proof}
Let $x\in\big[\FU(2n,t^pA,t^pI,t^p\Gamma),\FU(2n,t^pA,t^pJ,t^p\Delta)]\big]$,
$y\in \EU(2n, \FormR)$ and $z\in \FU^1(2n,\frac{K}{t^{m}},\frac{\Omega}{t^{m}})$.
Then using  the Hall--Witt identity (C3), one obtain that
\begin{equation} \Label{eqn:99}
\big[x,[y^{-1},z]\big]=  {}^{y^{-1}x}\big[[x^{-1},y],z\big]\, \, \,
{}^{y^{-1}z}\big[[z^{-1},x],y\big].
\end{equation}
By our assumption and Lemma~\ref{Lem:Habdank}, we have
\begin{align*}
x\in \big[\FU(2n,t^pA,t^pI,t^p\Gamma),& \FU(2n,t^pA,t^pJ,t^p\Delta)]\big]\\
&\subseteq \big[\EU(2n,t^pI,t^p\Gamma), \EU(2n,t^pJ,t^p\Delta)]\big]\\
&\subseteq \GU(2n, t^{2p}(I\circ J), t^{2p}(\Gamma\circ\Delta)).
\end{align*}
Using the commutator formula,  we obtain that
\begin{align*}
[x^{-1},y]&\in\Big[\GU(2n,t^{2p}(I\circ J),t^{2p}(\Gamma\circ\Delta)),
\EU(2n,\FormR)\Big]\\
&=\EU(2n, t^{2p}(I\circ J), t^{2p}(\Gamma\circ\Delta))\\
&=\EU\big(2n,t^{2p}(I\circ J),
\Lambda \circ t^{2p}(\Gamma\circ\Delta)\big).
\end{align*}
\noindent
By Lemma~\ref{Lem:Habdank}, it follows that
\begin{align*}
\EU\big(2n,t^{2p}(I\circ J)&,\Lambda\circ
t^{2p}(\Gamma\circ\Delta)\big )\\
&\subseteq \Big[\FU(2n,t^pA,t^p\Lambda), \FU(2n, t^p(I\circ J),
t^p(\Gamma\circ\Delta)\Big]\\
&=\FU(2n,t^pA,  t^p(I\circ J), t^p(\Gamma\circ\Delta)).
\end{align*}
\noindent
Therefore Lemma~\ref{Lem:RVZ4} implies that for any given $p'$
there exists an integer $p$
such that
\begin{align*}
\big[[x^{-1},y],z\big]&\in\ {}^{y^{-1}x}\Big[\FU(2n,t^p A,
t^p(I\circ J), t^p(\Gamma\circ\Delta)),
\FU^1(2n,\frac{K}{t^{m}},\frac{\Omega}{t^{m}})\Big]\\
&\subseteq\ {}^{y^{-1}x}\Big[ \FU(2n, t^{p'}A, t^{p'}(I\circ J),
t^{p'}(\Gamma\circ\Delta)), \FU(2n, t^{p'}A, t^{p'}K, t^{p'}\Omega)\Big],
\end{align*}
\noindent
where by definition, ${y^{-1}x}\in \EU(2n,\frac{A}{t^0}, \frac{\Lambda}{t^0})$.
By (\ref{eqn:RVZ2}) in Lemma~\ref{Lem:cong2}, for any given $l$, we may find
a sufficiently large $p'$ such that
\begin{multline*}
{}^{y^{-1}x}\Big[\FU(2n,t^{p'}A,t^{p'}(I\circ J),
t^{p'}(\Gamma\circ\Delta)),
\FU(2n, t^{p'}A, t^{p'}K, t^{p'}\Omega)\Big]\\
\subseteq\Big[\FU(2n,t^{2l}A,t^{2l}(I\circ J),
t^{2l}(\Gamma\circ\Delta)),
\FU(2n, t^{2l}A, t^{2l}K, t^{2l}\Omega)\Big].
\end{multline*}
\noindent
Now applying Lemma~\ref{Lem:Habdank} again, we have
\begin{align*}
{}^{y^{-1}x}\big[[x^{-1},y],z\big] &\in
\Big[ \FU(2n, t^{2l}A, t^{2l}(I\circ J), t^{2l}(\Gamma\circ\Delta)),
\FU(2n, t^{2l}A, t^{2l}K, t^{2l}\Omega)\Big]\\
&\subseteq
\Big[ \big[ \FU(2n, t^{l}A, t^lI, t^l\Gamma),
\FU(2n, t^lA, t^lJ, t^l\Delta)\big],
\FU(2n, t^{2l}A, t^{2l}K, t^{2l}\Omega)\Big]\\
&\subseteq\Big[ \big[ \FU(2n, t^{l}A, t^lI, t^l\Gamma),
\FU(2n, t^lA, t^lJ, t^l\Delta)\big], \FU(2n, t^{l}A, t^{l}K, t^{l}\Omega)\Big].
\end{align*}
\noindent
This proves the first factor of (\ref{eqn:99}) satisfies the lemma.
\par
For the second factor, using Lemma~\ref{Lem:Habdank}
\begin{align*}
[z^{-1},x] &\in \Big[\FU^1(2n,\frac{K}{t^{m}},\frac{\Omega}{t^{m}}),
[\FU(2n,t^pA,t^pI,t^p\Gamma), \FU(2n,t^pA,t^pJ,t^p\Delta)]\Big]\\
&\subseteq \Big[\FU^1(2n,\frac{K}{t^{m}},\frac{\Omega}{t^{m}}),
\GU(2n,t^{2p}(I\circ J),t^{2p}(\Gamma\circ\Delta))]\Big].
\end{align*}
\noindent
Now applying Lemma~\ref{Lem:EGG}, for any given integers $l$ and $m$,
there exists an integer $p$ such that
\begin{align*}
\Big[\FU^1(2n,\frac{K}{t^{m}},\frac{\Omega}{t^{m}}),
&\GU(2n,t^p(I\circ J),t^p(\Gamma\circ\Delta))]\Big]\\
&\subseteq \GU\big(2n, t^{3p'}(K\circ (I\circ J))\big),
t^{3p'}(\Omega\circ(\Gamma\circ\Delta))).
\end{align*}
\noindent
Again by the commutator formula, we obtain
\begin{align*}
\big[[z^{-1},x],y\big]&\in\Big[\GU\big(2n, t^{3p'}(K\circ (I\circ J))\big),
t^{3p'}(\Omega\circ(\Gamma\circ\Delta)) ), \EU(2n, \FormR)]\\
&=\EU\big(2n, t^{3p'}(K\circ (I\circ J))\big),
t^{3p'}(\Omega\circ(\Gamma\circ\Delta)) ).
\end{align*}
\noindent
Applying Lemma~\ref{Lem:Habdank} twice, we get
\begin{align*}
\EU\big(2n, t^{3p'}(K\circ &(I\circ J)),
t^{3p'} (\Omega\circ(\Gamma\circ\Delta) )\big)\\
&\subseteq \Big[\big[\FU(2n,t^{p'}I,t^{p'}\Gamma),
\FU(2n, t^{p'}J,t^{p'}\Delta)\big],\FU(2n, t^{p'}K,t^{p'}\Omega)\Big].
\end{align*}
\noindent
Finally, we have
\begin{multline*}
{}^{y^{-1}z}\big[[z^{-1},x],y\big]\in\ {}^{y^{-1}z}\Big[
\big[\FU(2n,t^{p'}I,t^{p'}\Gamma), \FU(2n, t^{p'}J,t^{p'}\Delta)\big],
\FU(2n, t^{p'}K,t^{p'}\Omega)
\Big]\\
=\Big[
\big[\ {}^{y^{-1}z}\FU(2n,t^{p'}I,t^{p'}\Gamma),\
{}^{y^{-1}z}\FU(2n, t^{p'}J,t^{p'}\Delta)\big], \
{}^{y^{-1}z}\FU(2n, t^{p'}K,t^{p'}\Omega)\Big].
\end{multline*}
\noindent
Now applying (\ref{eqn:RVZ2}) in Lemma~\ref{Lem:cong2} to every component of
the commutator above, we may find a sufficiently
large $p'$ such that for any given $l$,
\begin{multline*}
\Big[
\big[\ {}^{y^{-1}z}\FU(2n,t^{p'}I,t^{p'}\Gamma),\
{}^{y^{-1}z}\FU(2n, t^{p'}J,t^{p'}\Delta)\big],\
{}^{y^{-1}z}\FU(2n, t^{p'}K,t^{p'}\Omega)
\Big]\\
\subseteq\Big[\big[\FU(2n,t^lA ,t^lI,t^l\Gamma),
\FU(2n,t^lA ,t^lJ,t^l\Delta)\big], \FU(2n,t^lA, t^lK,t^l\Omega)\Big].
\end{multline*}
\noindent
This finishes the proof.
\end{proof}


\section{Main lemma on triple commutators}\label{sec8}

The following lemma is crucial for proving the main result, i.e.,
Theorem~\ref{main_triple} of this paper.

\begin{Lem}\Label{Lem:08}
Let $(\FormR)$ be a module finite form algebra, $(I,\Gamma)$,
$(J,\Delta)$ and $(K,\Omega)$  form ideals of  $(\FormR)$, and
$t\in R_0$. For any given $e_2\in \EU(2n,K_t, \Omega_t)$ and
integer $l$, there is a  sufficiently large integer $p$, such that
\begin{equation}\label{eqn:l1}
[e_1,e_2]\in \Big[\big[\EU(2n,t^lI,t^l\Gamma), \EU(2n,t^lJ,t^l\Delta)\big],
\EU(2n,t^lK,t^l\Omega)\Big],
\end{equation}
where  $e_1\in [\FU^1(2n,t^pI,t^p\Gamma), \EU(2n, J,\Delta)]$.
\end{Lem}
\begin{proof}
For any given $e_1 \in [\FU^1(2n,t^pI,t^p\Gamma), \EU(2n, J,\Delta)]$
and $e_2\in \EU(2n, K_t,\Omega_t)$, one may find some positive
integers $m$, $L$ and $S$, such that
$$ e_1\in [\FU^1(2n,t^pI,t^p\Gamma), \FU^{S}(2n,A, J,\Delta)] $$
\noindent
and
$$ e_2\in \FU^L\big(2n,\frac{A}{t^m},\frac{K}{t^m}, \frac\Omega{t^m}\big). $$


Applying the identity (C1$^+$) and repeated application of
(\ref{eqn:RVZ1}) in Lemma~\ref{Lem:cong2}, we reduce the problem to show that
\begin{multline*}
\Big[[\FU^1(2n,t^pI,t^p\Gamma), \FU^S(2n,A,J,\Delta)],
{}^cT_{i,j}(\frac{\gamma}{t^m})\Big] \subseteq\\
\Big[\big[\EU(2n,t^lI,t^l\Gamma), \EU(2n,t^lJ,t^l\Delta)\big],
\EU(2n,t^lK,t^l\Omega)\Big],
\end{multline*}
\noindent
where $c\in\FU^1(2n,\frac{A}{t^m},\frac{\Lambda}{t^m})$ and
$T_{i,j}(\frac{\gamma}{t^m})\in \EU(2n,\frac{K}{t^m},\frac{\Omega}{t^m})$.
\par
We claim that for any given integer $p$, there exists some integer
$m'$ such that  any elementary root element $T_{i,j}(\frac{\gamma}{t^m})$
can be further decomposed as a product
$$ \Big[\FU^1(2n, {t^pA}, {t^p} \Lambda),
\FU^1(2n,\frac{K}{t^{m'}},\frac{\Omega}{t^{m'}})\Big]
\Big[\FU^1(2n, {t^pA}, {t^p} \Lambda),
\FU^1(2n,\frac{K}{t^{m'}},\frac{\Omega}{t^{m'}})\Big]. $$
\noindent
Suppose that $T_{i,j}(\frac{\gamma}{t^m})$ is a short root.
Let $k\ne\pm i, \pm j$. Then by (R4), we have
$$ T_{i,j}(\frac{\gamma}{t^{m}})=
\big[T_{i,k}(t^p), T_{k,j}( \frac{\gamma}{t^{m-p}})], $$
\noindent
which satisfies our claim.
\par
Suppose that $T_{i,j}(\frac{\gamma}{t^m})$ is a long root.
Let $k\ne\pm i $. Using a variation of (R6), we get
\begin{eqnarray*}
T_{i,-i}(\frac{\gamma}{t^m})&=&T_{i,-i}(t^{p}\frac{\gamma}{t^{m-2p}}t^p)\\
&=&T_{i,-k}(-t^p\lambda^{-(\ep(k)-\ep(i))/2}\frac{\gamma}{t^{m-2p}})
\Big[T_{i,k}(t^p),T_{k,-k}(\lambda^{-(\ep(k)-\ep(i))/2}\frac{\gamma}{t^{m-2p}})\Big].
\end{eqnarray*}
\noindent
By the previous paragraph, we have
$$ T_{i,-k}(-\lambda^{-(\ep(k)-\ep(i))/2}\frac{\gamma}{t^{m-p}})\in
\Big[\FU^1(2n, {t^pA}, {t^p} \Lambda),\FU^1(2n,\frac{K}{t^{m'}},
\frac{\Omega}{t^{m'}})\Big]. $$
\noindent
This proves the claim.
\par
Together with he identity  (C1$^+$)  and (\ref{eqn:RVZ3}) in
Lemma~\ref{Lem:cong2}, the claim above allows us further reduce
the proof  to show that
\begin{align*}
\Big[\big[\FU^1(2n,t^pI,t^p\Gamma),\FU^1(2n,&J,\Delta)\big],\
{}^c\big[\FU^1(2n, {t^pA}, {t^p} \Lambda),
\FU^1(2n,\frac{K}{t^{m'}},\frac{\Omega}{t^{m'}})\big]\Big]   \\
&\subseteq \Big[\big[\EU(2n,t^lI,t^l\Gamma), \EU(2n,t^lJ,t^l\Delta)\big],
\EU(2n,t^lK,t^l\Omega)\Big].
\end{align*}
\noindent
Clearly, $\FU^1(2n,J,\Delta)=\FU^1(2n,\frac {J}{t^0}, \frac\Delta{t^0})$.
By (\ref{eqn:RVZ3}) in Lemma~\ref{Lem:cong2}, for any given $p'$
we have an integer $p$ such that
$$ \big[\FU^1(2n,t^pI,t^p\Gamma), \FU^1(2n,J,\Delta)\big]\subseteq
\big[\FU(2n,t^{p'}A ,t^{p'}I,t^{p'}\Gamma),
\FU(2n,t^{p'}A ,t^{p'}J,t^{p'}\Delta)\big]. $$
\noindent
Therefore, we obtain
\begin{align*}
\Big[\big[\FU^1&(2n,t^pI,t^p\Gamma), \FU^S(2n,A,J,\Delta)\big],\
{}^c\big[\FU^1(2n, {t^pA}, {t^p} \Lambda),
\FU^1(2n,\frac{K}{t^{m'}},\frac{\Omega}{t^{m'}})\big]\Big]\\
&\subseteq
\Big[\big[\FU(2n,t^{p'}A ,t^{p'}I,t^{p'}\Gamma),
\FU(2n,t^{p'}A ,t^{p'}J,t^{p'}\Delta)\big],\\
&\hskip 2truein
{}^c\big[\FU^1(2n, {t^pA}, {t^p} \Lambda),
\FU^1(2n,\frac{K}{t^{m'}},\frac{\Omega}{t^{m'}})\big]\Big]\\
&\subseteq
\ {}^c\Big[\big[\FU(2n,t^{p'}A ,t^{p'}I,t^{p'}\Gamma),
\FU(2n,t^{p'}A ,t^{p'}J,t^{p'}\Delta)\big]^c,\\
&\hskip 2truein\big[\FU^1(2n, {t^pA}, {t^p} \Lambda),
\FU^1(2n,\frac{K}{t^{m'}},\frac{\Omega}{t^{m'}})\big]\Big].\\
\end{align*}
\noindent
Applying (\ref {eqn:RVZ2}) in Lemma~\ref{Lem:cong2}, we find
an integer $p'$ such that
\begin{align*}
\big[\FU(2n,t^{p'}A ,t^{p'}I,t^{p'}\Gamma),
&\FU(2n,t^{p'}A ,t^{p'}J,t^{p'}\Delta)\big]^c\\
&\subseteq \big[\FU(2n,t^{p''}A ,t^{p''}I,t^{p''}\Gamma),
\FU(2n,t^{p''}A ,t^{p''}J,t^{p''}\Delta)\big]
\end{align*}
for any given integer $p''$.
Thanks to Lemma~\ref{Lem:FFEF}, for any given $l'$, we find
an integer $p''$ such that
\begin{align*}
\Big[\big[&\FU(2n,t^{p''}A ,t^{p''}I,t^{p''}\Gamma),
\FU(2n,t^{p''}A ,t^{p''}J,t^{p''}\Delta)\big],\\
&\hskip 2truein
\big[\FU^1(2n, {t^pA}, {t^p} \Lambda),
\FU^1(2n,\frac{K}{t^{m'}},\frac{\Omega}{t^{m'}})\big]\Big]\\
&\subseteq \ {}^c\Big[\big[\FU(2n,t^{l'}A ,t^{l'}I,t^{l'}\Gamma),
\FU(2n,t^{l'}A ,t^{l'}J,t^{l'}\Delta)\big],
\FU(2n,t^{l'}A, t^{l'}K,t^{l'}\Omega)\Big]\\
&=\Big[\big[\ {}^c\FU(2n,t^{l'}A ,t^{l'}I,t^{l'}\Gamma), \
{}^c\FU(2n,t^{l'}A ,t^{l'}J,t^{l'}\Delta)\big], \
{}^c\FU(2n,t^{l'}A, t^{l'}K,t^{l'}\Omega)\Big].
\end{align*}
\noindent
Applying (\ref{eqn:RVZ1}) in Lemma~\ref{Lem:cong2} to each component
of the commutator above, we may find a sufficiently large integer $l'$
such that for any given $l$
\begin{align*}
\Big[\big[&\ {}^c\FU(2n,t^{l'}A ,t^{l'}I,t^{l'}\Gamma), \
{}^c\FU(2n,t^{l'}A ,t^{l'}J,t^{l'}\Delta)\big], \
{}^c\FU(2n,t^{l'}A, t^{l'}K,t^{l'}\Omega)\Big]\\
&\subseteq
\Big[\big[\EU(2n,t^lI,t^l\Gamma), \EU(2n,t^lJ,t^l\Delta)\big],
\EU(2n,t^lK,t^l\Omega)\Big].
\end{align*}
\noindent
Hence,
$$ [e_1,e_2]\in\Big[\big[\EU(2n,t^lI,t^l\Gamma),
\EU(2n,t^lJ,t^l\Delta)\big], \EU(2n,t^lK,t^l\Omega)\Big], $$
\noindent
which finishes the proof.
\end{proof}


\section{Proof of Theorem~\ref{main_triple}}\label{sec9}

Now we are all set to complete the proof of the {\it triple}
commutator formula, Theorem~\ref{main_triple}.
\par
The functors $\EU_{2n}$ and $\GU_{2n}$ commute with direct limits.
By~\S\ref{sub:1.3}, one reduces the proof  to the case where $A$ is finite
over $R_0$ and $R_0$ is Noetherian.
\par
By the relative standard commutator formula, Theorem~\ref{main_HVZ}, we have
\begin{equation*}
\big[\EU(2n,I,\Gamma),\GU(2n,J,\Delta)\big]=
\big[\EU(2n,I,\Gamma),\EU(2n,J,\Delta)\big].
\end{equation*}
\par
Thus, to prove Theorem~\ref{main_triple} it suffices to prove the following equality
\begin{multline*}
\Big[\big[\EU(2n,I,\Gamma),\EU(2n,J,\Delta)\big],\GU(2n,K,\Omega)\Big]=\\
\Big[\big[\EU(2n,I,\Gamma),\EU(2n,J,\Delta)\big],\EU(2n,K,\Omega)\Big].
\end{multline*}
\noindent
By Theorem~\ref{Comgenerator} the mixed commutator subgroup
$[\EU(2n,I,\Gamma),\EU(2n,J,\Delta)\big]$ is generated by the
conjugates in $\EU(2n,\FormR)$ of the following types of elements
\begin{equation}\Label{ggsswwi}
e=\big[T_{ji}(\alpha),{}^{T_{ij}(a)}T_{ji}(\beta)\big],\qquad
e=\big[T_{ji}(\alpha),T_{ij}(\beta)\big],
\qquad \text{ and} \qquad e=T_{ij}(\xi),
\end{equation}
where $\alpha\in(I,\Gamma)$, $\beta\in(J,\Delta)$,
$\xi\in(I,\Gamma)\circ(J,\Delta)$ and $a\in(A,\Lambda)$.
\par
We claim that for any $g\in \GU(2n,K,\Omega)$,
\begin{equation}\label{jhgfsap}
\big[e,g]\in\Big[\big[\EU(2n,I,\Gamma),\EU(2n,J,\Delta)\big],\EU(2n,K,\Omega)\Big].
\end{equation}

Let  $g\in \GU(2n,K,\Omega)$.
For any maximal ideal $\gm\in\Max(R_0)$, the form ring
$(A_\gm,\Lambda_\gm)$ contains $(K_\gm,\Omega_\gm)$ as a form ideal.
Consider the localisation homomorphism $F_\gm:A\rightarrow A_\gm$ which induces
homomorphisms on the level of unitary groups,
$$ F_\gm:\GU(2n,A,\Lambda)\rightarrow\GU(2n,A_\gm,\Lambda_\gm), $$
\noindent
and
$$ F_\gm:\GU(2n,K,\Omega)\rightarrow\GU(2n,K_\gm,\Omega_\gm). $$
\par
Therefore, for $g\in\GU(2n,K,\Omega)$,
$F_\gm(g)\in\GU(2n,K_\gm,\Omega_\gm)$.
Since $A_\gm$ is module finite over the local ring $R_\gm$, $A_\gm$ is semi-local
\cite[III(2.5), (2.11)]{Bass1}, therefore its stable rank is $1$. It follows by
(see \cite[9.1.4]{HO}) that,
$$ \GU(2n,K_\gm,\Omega_\gm)=\EU(2n,K_\gm,\Omega_\gm)\GU(2,K_\gm,\Omega_\gm). $$
\par
Thus, $F_\gm(g)$ can be decomposed as $F_\gm(g)=\ep h$, where
$\ep\in\EU(2n,K_\gm,\Omega_\gm)$ and
$h\in\GU(2,K_\gm,\Omega_\gm)$ is a
$2\times2$ matrix embedded in $\GU(2n,K_\gm,\Omega_\gm)$ and this
embedding can be arranged modulo $\EU(2n,K_\gm,\Omega_\gm)$.
\par
Now, by (\ref{sub:1.3}), we may reduce the problem to the case $A_t$ with
$t\in R_0\backslash \gm$. Namely,
\begin{equation}\label{ppooi}
F_t(g)=\ep h,
\end{equation}
where $\ep\in\EU(2n,K_t,\Omega_t)$ and $h\in\GL(2,K_t,\Omega_t)$.

For any maximal ideal $\gm\lhd R_0$, choose  $t_{\gm} \in R_0\backslash \gm$
as above and an arbitrary positive integer $p_{\gm}$. (We will later
choose $p_{\gm}$ according to Lemma~\ref{Lem:08}.)
Since the collection of all $\{t_\gm^{p_\gm} \mid \gm \in \max(R_0) \}$
is not contained in any maximal ideal, we may find a finite number of
$t_{\gm_s}^{p_s}\in R_0\backslash \gm_s$ and $x_s\in R_0$, $s=1,\dots,k$,
such that
$$ \sum_{s=1}^k t_{\gm_s}^{p_s}x_s=1. $$
\par
In order to prove~(\ref{jhgfsap}), first we consider the generators of
the first kind in~(\ref{ggsswwi}), namely
\[e=\Big[T_{ji}(\alpha),{}^{T_{ij}(a)}T_{ji}(\beta)\Big].\]
Consider
\begin{multline*}
e=\big[T_{ji}(\alpha),{}^{T_{ij}(a)}T_{ji}(\beta)\big]=
\Bigg[T_{ji}\bigg(\sum_{s=1}^kt_{\gm_s}^{p_s}x_s\alpha\bigg),
{}^{T_{ij}(a)}T_{ji}(\beta)\Bigg]=\\
\Bigg[\prod_{s=1}^kT_{ji}(t_{\gm_s}^{p_s}x_s\alpha),
{}^{T_{ij}(a)}T_{ji}(\beta)\Bigg].
\end{multline*}
\noindent
By Identity (C$2^+$), the element
$e=\displaystyle \Bigg[\prod_{s=1}^kT_{ji}(t_{\gm_s}^{p_s}x_s\alpha),
{}^{T_{ij}(a)}T_{ji}(\beta)\Bigg]$ can be written as a product of
the following form:
\begin{multline}\label{hfsis}
e={}^{T_k}\Big[T_{ji}(t_{\gm_k}^{p_k}x_k\alpha),
{}^{T_{ij}(a)}T_{ji}(\beta)\Big]\cdot
{}^{T_{k-1}}\Big[T_{ji}(t_{\gm_{k-1}}^{p_{k-1}}x_{k-1}\alpha),
{}^{T_{ij}(a)}T_{ji}(\beta)\Big]\cdot\\ \ldots \cdot
{}^{T_1}\Big[T_{ji}(t_{\gm_1}^{p_1}x_1\alpha),
{}^{T_{ij}(a)}T_{ji}(\beta)\Big],
\end{multline}
\noindent
where $T_1,T_2,\ldots, T_k\in \EU(2n,A,\Lambda)$. Note that from
(C$2^+$) it is clear that all $T_s$, $s=1,\dots,k$, are products
of elementary matrices of the form $T_{ji}(a)$. Thus $T_s=T_{ji}(a_s)$,
where $a_s\in A$ and $s=1,\dots,k$, which clearly commutes with
$T_{ji}(x)$ for any $x\in A$. So the commutator~(\ref{hfsis}) is equal to
\begin{multline}\label{hfsis1}
e=\Big[T_{ji}(t_{\gm_k}^{p_k}x_k\alpha),
{}^{T_k}{}^{T_{ij}(a)}T_{ji}(\beta)\Big]\cdot
\Big[T_{ji}(t_{\gm_{k-1}}^{p_{k-1}}x_{k-1}\alpha),
{}^{T_{k-1}}{}^{T_{ij}(a)}T_{ji}(\beta)\Big]\cdot\\
\ldots\cdot \Big[T_{ji}(t_{\gm_1}^{p_1}x_1\alpha),
{}^{T_1}{}^{T_{ij}(a)}T_{ji}(\beta)\Big].
\end{multline}


Using (C$2^+)$ and in view of~(\ref{hfsis1}) we obtain that $[e,g]$ is a
product of the conjugates in $\EU(2n,A,\Lambda)$ of
\[w_s=\bigg[\Big[T_{ji}(t_{\gm_s}^{p_{s}}x_{s}\alpha),
{}^{T_{ji}(a_s)T_{ij}(a)}T_{ji}(\beta)\Big],g\bigg],\]
\noindent
where $a_s \in A$ and $s=1,\ldots,k$.
\par
For each $s=1,\dots,k$, consider $\theta_{t_{\gm_s}}(w_s)$ which we
still write as $w_s$ but keep in mind that this image
is in $\GU(2n, A_{t_{\gm_s}},\Lambda_{t_{\gm_s}})$.
\par
Note that all $\Big[T_{ji}(t_{\gm_s}^{p_{s}}x_{s}\alpha),
{}^{T_{ji}(a_s)T_{ij}(a)}T_{ji}(\beta)\Big]$, $s=1,\ldots,k$,
differ from the identity matrix at only the rows $\pm i,\pm j$
and in the columns $\pm i,\pm j$. Since $n>2$, we can choose an $h$
in the decomposition~(\ref{ppooi}) so that it commutes with
\[\Big[T_{ji}(t_{\gm_s}^{p_{s}}x_{s}\alpha),
{}^{T_{ji}(a_s)T_{ij}(a)}T_{ji}(\beta)\Big].\]
\noindent
This allows us to reduce $\theta_{t_{\gm_s}}(w_s)$  to
$$ \bigg[\Big[T_{ji}(t_{\gm_s}^{p_{s}}x_s\alpha),
{}^{T_{ji}(a_s)T_{ij}(a)}T_{ji}(\beta)\Big], \ep\bigg], $$
\noindent
where $\ep \in E_n(A_{t_{\gm_s}},K_{t_{\gm_s}})$.
By Lemma~\ref{Lem:08}, for any given $l_{s}$, there is a sufficiently
large $p_{s}$, $s=1,\dots k$, such that
\begin{multline*}
\bigg[\Big[T_{ji}(t_{\gm_s}^{p_{s}}x_s\alpha),
{}^{T_{ji}(a_s)T_{ij}(a)}T_{ji}(\beta)\Big], \ep\bigg]
\in \\
\Big[\big[\EU(2n,t^{l_s}I,t^{l_s}\Gamma),
\EU(2n,t^{l_s}J,t^{l_s}\Delta)\big],\EU(2n,t^{l_s}K,t^{l_s}\Omega)\Big].
\end{multline*}
\noindent
Let us choose $l_{s}$ to be large enough so that by Lemma~\ref{Lem:03}
the restriction of
$$ \theta_{t_{\gm_s}}:\GL_n(A,t_{\gm_s}^{l_{s}}A)\to\GL_n(A_{t_{\gm_s}}) $$
\noindent
be injective. Then it is easy to see that for any $s$, we have
$$ \bigg[\Big[T_{ji}(t_{\gm_s}^{p_{s}}x_s\alpha),
{}^{T_{ji}(a_s)T_{ij}(a)}T_{ji}(\beta)\Big],g\bigg]\in
\Big[\big[\EU(2n,I,\Gamma),\EU(2n,J,\Delta)\big],\EU(2n,K,\Omega)\Big]. $$
\noindent
Since relative elementary subgroups are normal in $\GU(2n,A,\Lambda)$
(Theorem~\ref{main_BV}), it follows that
$$ [e,g]\in \Big[\big[\EU(2n,I,\Gamma),\EU(2n,J,\Delta)\big],
\EU(2n,K,\Omega)\Big]. $$
\par
When the generator is of the second kind, $e=[T_{ij}(\alpha),T_{ji}(\beta)]$,
a similar argument goes through, which is left to the reader.
\par
Now consider the generators of the 3rd  kind, namely, the conjugates
of the following type of elements, $e=T_{ij}(\alpha\beta)$. By the
normality of $\EU(2n,(I,\Gamma)\circ(J,\Delta))$, the
conjugates of $e$ are in $\EU(2n,(I,\Gamma)\circ(J,\Delta))$. We have
$$ [e,g]\in \big[\EU(2n,(I,\Gamma)\circ(J,\Delta)),
\GU(2n,K,\Omega)\big]. $$
\noindent
By the generalized commutator formula (Theorem~\ref{main_HVZ}), one obtains
$$ \big[\EU(2n,(I,\Gamma)\circ(J,\Delta)),\GU(2n,K,\Omega)\big]=
\big[\EU(2n,(I,\Gamma)\circ(J,\Delta)),\EU(2n,K,\Omega)\big]. $$
\noindent
Now applying Lemma~\ref{Lem:Habdank}, we finally get
$$ \big[\EU(2n,(I,\Gamma)\circ(J,\Delta)),\EU(2n,K,\Omega)\big]\subseteq
\Big[\big[\EU(2n,I,\Gamma),\EU(2n,J,\Delta)\big],\EU(2n,K,\Omega)\Big]. $$
\noindent
Therefore,
$[e,g]\in\Big[\big[\EU(2n,I,\Gamma),\EU(2n,J,\Delta)\big],\EU(2n,K,\Omega)\Big]$.
This proves our claim. Thus we established~(\ref{jhgfsap}) for all type
of generators $e$ of~(\ref{ggsswwi}).
\par
To finish the proof, let
$e\in\big[\EU(2n,I,\Gamma),\GU(2n,J,\Delta)\big]=
\big[\EU(2n,I,\Gamma),\EU(2n,J,\Delta)\big]$, and $g\in\GU(2n,K,\Omega)$.
Then by Lemma~\ref{Comgenerator},
$$ e={}^{c_1}e_{1}\times{}^{c_2}e_{2}\times\cdots\times {}^{c_k}e_{k}, $$
\noindent
with $c_i\in\EU(2n,A,\Lambda)$ and $e_{i}$ takes any of the forms
in~(\ref{ggsswwi}). Since the relative elementary subgroups are normal,
Identity (C$2^+$) implies that it suffices to show that
$$ [{}^{c_i}e_{i},g]\in
\Big[\big[\EU(2n,I,\Gamma),\EU(2n,J,\Delta)\big],\EU(2n,K,\Omega)\Big],
\qquad i=1,\dots,k. $$
\noindent
Now, since both the relative elementary subgroups and the principal
congruence subgroups are normal, Identity (C5) further reduces the
problem to verification of the
inclusions
$$ [e_i,g]\in
\Big[\big[\EU(2n,I,\Gamma),\EU(2n,J,\Delta)\big],\EU(2n,K,\Omega)\Big],
\qquad i=1,\dots,k. $$
\noindent
But this is exactly what has been shown above. This completes the proof
of Theorem~\ref{main_triple}, the rest is now an exercise.


\section{Multiple commutator formulas for group functors}\label{sec10}

Here we finish the proof of Theorems~\ref{main_multi} and ~\ref{main_multi2}. In fact, we verify that
these theorems formally follow from Theorems~\ref{main_HVZ} and \ref{main_triple} and Lemmas~\ref{Lem:Habdank}
and \ref{Lem:GGG}.
\par
Namely, let $G_0,\dots,G_n$, $n\geq 3$, be subgroups of a given group
$G$. There are many ways to arrange brackets
$[\,\underline{\ \ }\,,\,\underline{\ \ }\,]$ in the sequence
$G_0,\ldots,G_n$ to correctly define the multi-commutator of these
subgroups. For instance, for $n=4$, we can have the following two
arrangements $\Big[\big [G_0,[G_1,G_2]\big], G_3\Big]$ and
$\Big[\big[G_0,G_1\big],\big[G_2, G_3\big]\Big]$, among others.
It is classically known that overall there are
$$ c_n=\frac{1}{(n+1)}\binom{2n}{n} $$
\noindent
ways to arrange brackets to form a multi-commutator of $n+1$
subgroups, where $c_n$ is the Catalan number. Any such arrangement
correctly defining a multi-commutator will be denoted by
$$ \big\llbracket G_0,G_1,\ldots,G_m \big\rrbracket. $$
\noindent
This notation is introduced to distinguish general such arrangements
from the standard multi-commutator $[G_0,\ldots,G_n]$, which is
usually interpreted as the left-normed commutator,
$$ [G_0,\ldots,G_n]=
\Big[\dots\big[[G_0,G_1],G_2\big],\dots,G_n\Big]. $$
\par
Let us fix the axiomatic setting for the proof of Theorems~\ref{main_multi} and \ref{main_multi2}.
Let $A$ be a ring. For each two-sided ideal $I$ of $A$, let $E(I)$
and $G(I)$ be subgroups of $G=G(A)$ such that $E(I)$ is a normal
subgroup of $G(I)$. Assume, for any three two-sided ideals $I,J$
and $K$ of $A$ the following holds
\par\smallskip
(M1) \ $E(I)\subseteq E(J)$ and $G(I)\subseteq G(J)$,
\par\smallskip
(M2) \  $\big[E(I),G(J)\big]=\big[E(I),E(J)\big]$,
\par\smallskip
(M3) \ $\Big[\big[E(I),G(J)\big],G(K)\Big]=\Big[\big[E(I),E(J)\big],E(K)\Big]$,
\par\smallskip
(M4) \ $E(I\circ J)\subseteq\big[E(I),E(J)\big]\subseteq\big[E(I),G(J)\big]
\subseteq\big[G(I),G(J)\big]\subseteq G(I\circ J)$.
\par\smallskip\noindent
Recall, that here we denote by $I\circ J=IJ+JI$ the symmetrised product
of the ideals $I$ and $J$. This operation is not associative, so
when writing $I_0\circ\ldots\circ I_m$ we assume that this is
the left-normed product. In particular, $I\circ J\circ K=(I\circ J)\circ K$.

\begin{Exp} Let $A$ be a quasi-finite $R$-algebra. The main results
of~\cite{RHZZ2} show that for any two-sided ideal $I$ of $A$,
$E(I)=E_n(A,I)$ and $G(I)=\GL_n(A,I)$ satisfy Conditions (M1)--(M4).
\end{Exp}
\begin{Exp} Let $(\FormR)$ be a quasi-finite form ring and let
$E(I_i)=\EU(2n,I_i,\Gamma_i)$ and $G(I_i)=\GU(2n,I_i,\Gamma_i)$.
Then Lemma~\ref{Lem:Habdank}, Theorem~\ref{main_HVZ},  Lemma~\ref{Lem:GGG}
and Theorem~\ref{main_triple} show that Conditions (M1)--(M4) are satisfied
in this setting.
\end{Exp}
The following lemma proves Theorem~\ref{main_multi}.
\begin{Lem}\label{comain}
Let $A$ be an $R$-algebra, $I_i $, $i=0,...,m$, be two-sided ideals
of $A$. Assume that subgroups $G(I)$ and $E(I)$ satisfy Conditions
{\rm (M1)--(M4)}. Then
\begin{equation}\label{corll11}
\big[E(I_0),G(I_1),G(I_2),\ldots,G(I_m)\big]=
\big[E(I_0),E(I_1),E(I_2),\ldots,E(I_m)\big].
\end{equation}
\end{Lem}
\begin{proof}
The proof proceeds by induction on $m$. For $i=1$ this is
Condition (M2). For $i=2$, this is Condition (M3) which will be
the first step of induction. Suppose the statement is valid for
$m-1$, when there are $m$ ideals in the commutator formula.
By Condition (M3), we have
\begin{equation*}
\bigg[\Big[\big[E(I_0),G(I_1)\big],G(I_2)\Big],G(I_3),\ldots,G(I_m)\bigg]=
\bigg[\Big[\big[E(I_0),E(I_1)\big],E(I_2)\Big],G(I_3),\ldots,G(I_m)\bigg].
\end{equation*}
On the other hand, by Condition (M4) one has
$[E(I_0),E(I_1)]\subseteq G(I_0I_1+I_1I_0)$. Thus
\begin{equation*}
\bigg[\Big[\big[E(I_0),E(I_1)\big],E(I_2)\Big],G(I_3),\ldots,G(I_m)\bigg]
\subseteq
\bigg[\Big[G(I_0I_1+I_1I_0),E(I_2)\Big],G(I_3),\ldots,G(I_m)\bigg].
\end{equation*}
\noindent
Since there are $m$ ideals involved in the commutator subgroups on the
right hand side, we can apply the induction hypothesis and get
\begin{equation*}
\bigg[\Big[G(I_0I_1+I_1I_0),E(I_2)\Big],G(I_3),\ldots,G(I_m)\bigg]=\\
\bigg[\Big[E(I_0I_1+I_1I_0),E(I_2)\Big],E(I_3),\ldots,E(I_m)\bigg].
\end{equation*}
\noindent
Finally, invoking Condition (M4) once more, we get $E(I_0I_1+I_1I_0)\subseteq
\big[E(I_0),E(I_1)\big]$. Substituting this inclusion in the above equality
we can conclude that the left hand side of~(\ref{corll11}) is contained in
the right hand side. Since the opposite inclusion is obvious, this completes
the proof.
\end{proof}

Now, we can go one step further, and show that in fact it does not matter
where the elementary subgroup appears in the multiple commutator
formula.

\begin{Lem}\Label{main2}
Let $A$ be an $R$-algebra and $I_i $, $i=0,...,m$, be two-sided ideals
of $A$. Assume that subgroups $G(I)$ and $E(I)$ satisfy Conditions
{\rm (M1)--(M4)}. Let $G_i$ be subgroups of $G(A)$ such that
$$ E(I_i)\subseteq G_i\subseteq G(I_i),\quad\text{ for } i=0,\ldots, m. $$
\noindent
If there is an index $j$ such that $G_j=E(I_j)$, then
\begin{equation}\label{comain2}
\big[G_0,G_1,\ldots,G_m\big]=\Big[E(I_0),E(I_1),E(I_2),\ldots,E(I_m)\Big].
\end{equation}
\end{Lem}
\begin{proof} For brevity, denote $E(I_i)$ by $E_i$. For a fixed
$m$ the proof proceeds on induction on $j$. As the base of induction
one takes $j=0$, which is the previous lemma. When $j=1$ one has
$\big[G_0,E_1\big]=\big[E_1, G_0\big]$, so this case reduces to the
case of $j=0$.
\par
For $2\le j\le m$ we can argue as follows. By assumption
\begin{multline*}
\big[G_0,G_1,\ldots,G_j,G_{j+1},\ldots, G_m\big]=
\Big[\big[G_0,G_1,\ldots G_j\big],G_{j+1},\ldots,G_m\Big]=\\
\Big[\big[G_0,G_1,\ldots,G_{j-1},E_j\big],G_{j+1},\ldots,G_m\Big]=
\Big[\big[[G_0,G_1,\ldots,G_{j-1}],E_j\big],G_{j+1},\ldots,G_m\Big].
\end{multline*}
\noindent
Now, repeated application of the rightmost inclusion from
Condition (M4) shows that
$$ \big[G_0,G_1,\ldots,G_{k-1}\big]\subseteq G(I_0\circ\ldots\circ I_{k-1}). $$
\noindent
Combining Conditions (M2) and (M4), we get
\begin{multline*}
\Big[\big[E_0,E_1,\ldots,E_{k-1}\big],E_k\Big]\subseteq
\Big[\big[G_0,G_1,\ldots G_{k-1}\big],E_k\Big]\\
\subseteq\big [G(I_0\circ\ldots\circ I_{k-1}), E_k\big]=
\big [E(I_0\circ\ldots\circ I_{k-1}), E_k\big]\subseteq
\Big [\big [E_0,E_1,\ldots, E_{k-1}\big], E_k\Big],
\end{multline*}
\noindent
and thus
\begin{equation*}
\big [G_0,G_1,\ldots,G_{k-1},E_k\big]=
\Big [\big [E_0,E_1,\ldots,E_{k-1}\big],E_k\Big].
\end{equation*}
\noindent
Substituting this into our commutator, we see that
\begin{equation*}
\big [G_0,G_1,\ldots,G_m\big]=
\big [E_0,E_1,\ldots,E_{k-1},E_k,G_{k+1},\ldots,G_{m}\big],
\end{equation*}
\noindent
and it only remains to invoke the previous lemma.
\end{proof}

Now we are all set for the final round of computation, to show
that it does not matter how the brackets are arranged either.
In particular, this proves Theorem~\ref{main_multi2}.

\begin{Lem}\Label{main3}
Let $A$ be an $R$-algebra and $I_i $, $i=0,...,m$, be two-sided ideals
of $A$. Assume that subgroups $G(I)$ and $E(I)$ satisfy Conditions
{\rm (M1)--(M4)}. Let $G_i$ be subgroups of $G(A)$ such that
$$ E(I_i)\subseteq G_i\subseteq G(I_i),\quad\text{ for } i=0,\ldots, m. $$
\noindent
If there is an index $j$ such that $G_j=E(I_j)$, then
\begin{equation}\label{hhalk}
\big\llbracket G_0,G_1,\ldots,G_m\big\rrbracket=
\big\llbracket E(I_0),E(I_1),\ldots,E(I_m)\big\rrbracket.
\end{equation}
\end{Lem}
\begin{proof}
To prove (\ref{hhalk}), we proceed by induction on $m$.
For $m=0$ and $m=1$ there is nothing to prove. For $m=2$, the
commutator $\big\llbracket G_0,G_1,G_2\big\rrbracket$ can
be arranged in six possible ways
\begin{multline*}
\Big[\big[G_0,G_1\big],E_2\Big],\qquad
\Big[E_0,\big[G_1,G_2\big]\Big],\qquad
\Big[\big[E_0,G_1\big],G_2\Big],\\
\Big[\big[G_0,E_1\big],G_2\Big],\qquad
\Big[G_0,\big[E_1,G_2\big]\Big],\qquad
\Big[G_0,\big[G_1,E_2\big]\Big],
\end{multline*}
\noindent
of which the first two and the last four are reduced to each other
by (C6), the commutativity of commutator on subgroups. Thus, it
only matters, whether $E_j$ stands inside the inner bracket, or outside
of it. The case, where it stands inside, was already considered
in the previous lemma, so that it only remains to consider the
first of the above arrangements. Using Conditions (M1)--(M4) we get
\begin{multline*}
\Big[\big[E_0,E_1\big],E_2\Big]\subseteq\Big[\big[G_0,G_1\big],E_2\Big]
\subseteq \Big[\big[G(I_0),G(I_1)\big],E(I_2)\Big]\\
\subseteq \Big[G(I_0\circ I_1),E(I_2)\Big]=
\Big[E(I_0\circ I_1),E(I_2)\Big]\\
\subseteq \Big[\big[E(I_0),E(I_1)\big],E(I_2)\Big ]=
\Big[\big[E_0,E_1\big],E_2\Big],
\end{multline*}
\noindent
as claimed.
\par
For the main step of induction, we consider two cases. Suppose first
there is  a mixed commutator $[G_i,G_{i+1}]$ in
$\big \llbracket G_0,G_1,\ldots, G_m\big \rrbracket$, where neither
$G_i$ nor $G_{i+1}$ is the {\it fixed} elementary subgroup $E_j$. Then
\begin{align}
\big\llbracket G_0,G_1,\ldots,G_m \big \rrbracket
& =\big\llbracket G_0,G_1,\ldots,[G_i,G_{i+1}],\dots,G_m\big\rrbracket\notag  \\
&\subseteq\big\llbracket G_0,G_1,\ldots,\big [G(I_i),G(I_{i+1})\big],\dots,
G_m \big \rrbracket \label{ryan1} \\
&\subseteq\big\llbracket G_0,G_1,\ldots,G(I_iI_{i+1}+I_{i+1}I_i),\dots,
G_m \big \rrbracket.\notag
\end{align}
\noindent
Note that there is one fewer ideal involved in the last commutator
formula (i.e., $m-1$ ideals) which also contains an elementary subgroup,
and so by induction
\begin{align}
\big \llbracket G_0,G_1,\ldots,& G(I_iI_{i+1}+I_{i+1}I_i),\dots,G_m\big\rrbracket
\notag\\
&= \big \llbracket E_0,E_1,\ldots,E(I_iI_{i+1}+I_{i+1}I_i),\dots, E_m\big \rrbracket
\notag \\
& \subseteq  \big \llbracket E_0,E_1,\ldots,\big[E(I_i),E(I_{i+1})\big],\dots,E_m\big
\rrbracket \label{ryan2} \\
&=  \big \llbracket E_0,E_1,\ldots,E_m \big \rrbracket.  \notag
\end{align}
\noindent
Putting~\ref{ryan1} and~\ref{ryan2} together, we get
$$ \big \llbracket G_0,G_1,\ldots,G_m \big \rrbracket =
\big \llbracket E_0,E_1,\ldots,E_m \big \rrbracket. $$
\par
It only remains to consider the case, where the only double mixed
commutator of the form $[G_i,G_{i+1}]$ inside our arrangement
$\big\llbracket G_0,G_1,\ldots,G_m\big\rrbracket$ involves our
fixed elementary subgroup $E_j$.
Consider the outermost pairs of inner brackets in our multicommutator
$$ \big \llbracket G_0,G_1,\ldots,G_m\big\rrbracket=
\Big[\big\llbracket G_0,G_1,\ldots,G_k\big\rrbracket, \big
\llbracket G_{k+1},\ldots,G_m\big\rrbracket\Big]. $$
\noindent
If $1\le k\le m-2$, then {\it each\/} of the inner brackets
$\big\llbracket G_0,G_1,\ldots,G_k\big\rrbracket$
and $\big\llbracket G_{k+1},\ldots,G_m\big\rrbracket$ contains
a double mixed commutator, one of which leaves $E_j$ outside,
and this is the situation we just considered.
\par
This leaves us with the analysis of the case, where $k=0$ or
$k=m-1$, in other words, $\llbracket G_0,G_1,\ldots,G_m\big\rrbracket$
is one of the following
$$ \Big[G_0,\big\llbracket G_{1},\ldots,G_m\big\rrbracket\Big],\qquad
\Big[\big\llbracket G_0,G_1,\ldots, G_{m-1}\big\rrbracket,G_m\Big]. $$
\noindent
By commutativity of the commutator, the first of these situations
reduces to the second one.
\par
Repeating this argument, we see that -- modulo the commutativity
of the commutator -- the arrangement
$\llbracket G_0,G_1,\ldots,G_m\big\rrbracket$ is the left-normed
one, for which the previous lemma already guarantees that
$$ \big\llbracket G_0,G_1,\ldots,G_m\big\rrbracket
=\big[G_{i_0},G_{i_1},\dots,G_{i_m}\big]
=\big[E_{i_0},E_{i_1},\dots,E_{i_m}\big]
=\big\llbracket E_0,E_1,\ldots,E_m\big\rrbracket. $$
\noindent
This finishes the proof of Lemma~\ref{main3} and thus also of Theorem~\ref{main_multi2}.
\end{proof}

Our final lemma proves Theorem~\ref{main_multi3}. Recall that the product of ideals are not associative. Thus in the following Lemma the bracketings
of the form ideals on the right hand side should correspond to the
bracketings of commutators on the left-hand side.

\begin{Lem}\Label{main4}
Let $A$ be an $R$-algebra and $I_i $, $i=0,...,m$, be two-sided ideals
of $A$. Assume that subgroups $G(I)$ and $E(I)$ satisfy Conditions
{\rm (M1)--(M4)}. Then
\begin{multline*}
\Big[\big\llbracket E(I_0),E(I_1),\ldots,E(I_{k})\big\rrbracket, \big
\llbracket E(I_{k+1}),\ldots,E(I_m)\big\rrbracket\Big]=\\
\big[E(I_0\circ\ldots\circ I_k\big),
E(I_{k+1}\circ\ldots\circ I_m)\big],
\end{multline*}
\noindent
where the bracketing of symmetrised products on the right hand side
coincides with the bracketing of the commutators on the left hand side.
\end{Lem}
\begin{proof}
Alternated application of (M4) and (M2) shows that
\begin{multline*}
\Big[\big\llbracket E(I_0),E(I_1),\ldots,E(I_{k})\big\rrbracket, \big
\llbracket E(I_{k+1}),\ldots,E(I_m)\big\rrbracket\Big]\le \\
\Big[G(I_0\circ\ldots\circ I_k\big),
\llbracket E(I_{k+1}),\ldots,E(I_m)\big\rrbracket\Big]=
\Big[E(I_0\circ\ldots\circ I_k\big),
\llbracket E(I_{k+1}),\ldots,E(I_m)\big\rrbracket\Big]\le \\
\big[E(I_0\circ\ldots\circ I_k\big),G(I_{k+1}\circ\ldots\circ I_m)\big]=
\big[E(I_0\circ\ldots\circ I_k\big),E(I_{k+1}\circ\ldots\circ I_m)\big]\le\\
\Big[\big\llbracket E(I_0),E(I_1),\ldots,E(I_{k})\big\rrbracket, \big
\llbracket E(I_{k+1}),\ldots,E(I_m)\big\rrbracket\Big],
\end{multline*}
\noindent
as claimed.
\end{proof}


\section{Final remarks}\label{sec11}

The present paper grew out of desire to prove the {\it general\/}
multiple commutator formula, which simultaneously generalises
our multiple commutator formula and nilpotent filtration of
relative $\K_1$, see \cite{BRN}. Of course, the general
commutator formula can only hold for finite-dimensional rings.
It can be stated as follows.

\begin{Prob}
Let $R$ be a ring of finite Bass--Serre dimension
$\delta(R)=d<\infty$, and let $(I_i,\Gamma_i)$, $1\le i\le m$, be
form ideals of $(R,\Lambda)$. Prove that for any $m\ge d$ one has
\begin{multline*}
\big[\GU(2n,I_0,\Gamma_0),\GU(2n,I_1,\Gamma_1),\GU(2n,I_2,\Gamma_2),
\ldots,\GU(2n,I_m,\Gamma_m)\big]=\\
=\big[\EU(2n,I_0,\Gamma_0),\EU(2n,I_1,\Gamma_1),\EU(2n,I_2,\Gamma_2),
\ldots,\EU(2n,I_m,\Gamma_m)\big].
\end{multline*}
\end{Prob}

In fact, recently we succeeded in proving such a formula for general
linear groups \cite{HSVZmult}. However the proof there critically
depends on a number of deep external results. In this respect the
case of unitary groups seems to be very different, since in the
context of unitary groups even the most basic results are simply
not there in the existing literature.
\par
For instance, the proof in \cite{HSVZmult} {\it starts\/} with the
following classical observation by Alec Mason and Wilson Stothers
\cite{Mason74}, which serves as the base of induction.

\begin{The}[{Mason--Stothers}]
Let $R$ be a ring, $I$ and $J$ be two two-sided ideals of $R$.
Assume that $n\ge\sr(R),3$. Then
$$ \big[\GL(n,R,I),\GL(n,R,J)\big]=\big[E(n,R,I),E(n,R,J)\big]. $$
\end{The}

For unitary groups, even such basic facts at the stable level
seem to be missing. After that the
proof in \cite{HSVZmult} proceeds by induction on $d$, which
depends on Bak's results \cite{B4}, precise form of injective
stability for $\K_1$, such as the Bass--Vaserstein theorem, etc.
It seems that to solve Problem~1 one has to rethink and expand
many aspects of structure theory of unitary groups, starting
with stability theorems \cite{BT,BPT,sinchuk}, more powerful
analogues of results on the superspecial unitary groups, than
what we established in \cite{RH,RH2,BRN}, etc. All these results
seem feasible, but to actually set them afoot might take a lot
of work.
\par
Let us mention two further problems closely related to the contents
of the present paper. Firstly, multiple commutator formulas are
relevant for the description of subnormal subgroups of
$\GU(2n,\FormR)$. Important progress in this direction was
recently obtained by the third author and You Hong \cite{ZZ1,youhong}.
But we feel that the bounds in these results can be improved
and hope to return to this problem with our new tools.
\par
Secondly, the generators constructed in Theorem~\ref{Comgenerator} are
a first approximation to the ``elementary'' generators of the
double commutator subgroups $[\EU(2n,I,\Gamma),\EU(2n,J,\Delta)]$.
Actually, building upon this theorem, we constructed a much
smaller set of generators, using which Alexei Stepanov was
able to prove finiteness results for relative commutators,
see \cite{yoga2,portocesareo}.
\par
We refer the interested reader to our forthcoming papers
\cite{yoga2,portocesareo,HSVZmult,HSVZ}, where these and some
other related problems are discussed in somewhat more detail.
\par
{\bf Acknowledgement.} We would like to thank Alexei Stepanov for close cooperation
on localisation methods over the last years. His contribution
to this circle of ideas was crucial, and in some cases decisive.
We thank Matthias Wendt who carefully studied \cite{RNZ1} and
suggested many improvements of the relative unitary commutator
calculus developed therein. Finally, we thank Rabeya Basu and
Ravi Rao who invited us to visit TIFR and IISER, for many
inspiring discussions and comparison of various localisation
methods.


\end{document}